\documentclass[11pt]{article}
\usepackage{latexsym,amsfonts,amssymb,amsmath,amsthm}

\usepackage{color,comment}

\begin{document}

\newtheorem{thm}{Theorem}[section]
\newtheorem{cor}{Corollary}[section]
\newtheorem{lem}{Lemma}[section]
\newtheorem{prop}{Proposition}[section]
\newtheorem{defn}{Definition}[section]
\newtheorem{rk}{Remark}[section]
\newtheorem{nota}{Notation}[section]
\newtheorem{Ex}{Example}[section]
\def\nm{\noalign{\medskip}}

\numberwithin{equation}{section}

\newcommand{\ds}{\displaystyle}
\newcommand{\pf}{\medskip \noindent {\sl Proof}. ~ }
\newcommand{\p}{\partial}
\renewcommand{\a}{\alpha}
\newcommand{\z}{\zeta}
\newcommand{\pd}[2]{\frac {\p #1}{\p #2}}
\newcommand{\norm}[1]{\left\| #1 \right \|}
\newcommand{\dbar}{\overline \p}
\newcommand{\eqnref}[1]{(\ref {#1})}
\newcommand{\na}{\nabla}
\newcommand{\Om}{\Omega}
\newcommand{\ep}{\epsilon}
\newcommand{\tmu}{\widetilde \epsilon}
\newcommand{\vep}{\varepsilon}
\newcommand{\tlambda}{\widetilde \lambda}
\newcommand{\tnu}{\widetilde \nu}
\newcommand{\vp}{\varphi}
\newcommand{\RR}{\mathbb{R}}
\newcommand{\CC}{\mathbb{C}}
\newcommand{\NN}{\mathbb{N}}
\renewcommand{\div}{\mbox{div}~}
\newcommand{\bu}{{\bf u}}
\newcommand{\la}{\langle}
\newcommand{\ra}{\rangle}
\newcommand{\Scal}{\mathcal{S}}
\newcommand{\Lcal}{\mathcal{L}}
\newcommand{\Kcal}{\mathcal{K}}
\newcommand{\Dcal}{\mathcal{D}}
\newcommand{\tScal}{\widetilde{\mathcal{S}}}
\newcommand{\tKcal}{\widetilde{\mathcal{K}}}
\newcommand{\Pcal}{\mathcal{P}}
\newcommand{\Qcal}{\mathcal{Q}}
\newcommand{\id}{\mbox{Id}}
\newcommand{\stint}{\int_{-T}^T{\int_0^1}}
%%%%%%%%%%

\newcommand{\be}{\begin{equation}}
\newcommand{\ee}{\end{equation}}

\newcommand{\rd}{{\mathbb R^d}}
\newcommand{\rr}{{\mathbb R}}
\newcommand{\alert}[1]{\fbox{#1}}
\newcommand{\eqd}{\sim}
\def\R{{\mathbb R}}
\def\N{{\mathbb N}}
\def\Q{{\mathbb Q}}
\def\C{{\mathbb C}}
\def\ZZ{{\mathbb Z}}
\def\l{{\langle}}
\def\r{\rangle}
\def\t{\tau}
\def\k{\kappa}
\def\a{\alpha}
\def\la{\lambda}
\def\De{\Delta}
\def\de{\delta}
\def\ga{\gamma}
\def\Ga{\Gamma}
\def\ep{\varepsilon}
\def\eps{\varepsilon}
\def\si{\sigma}
\def\Re {{\rm Re}\,}
\def\Im {{\rm Im}\,}
\def\E{{\mathbb E}}
\def\P{{\mathbb P}}
\def\Z{{\mathbb Z}}
\def\D{{\mathbb D}}
\def\p{\partial}
\newcommand{\ceil}[1]{\lceil{#1}\rceil}

\title{Logistic type attraction-repulsion chemotaxis systems with a free boundary or unbounded boundary.
I. Asymptotic dynamics in fixed unbounded domain}

\author{Lianzhang Bao\thanks{School of Mathematics, Jilin University, Changchun, 130012, P. R. China, and Department of Mathematics and Statistics,
Auburn University,  AL 36849, U. S. A. (lzbao@jlu.edu.cn), partially supported by the CPSF--183816.}\,\,   and
Wenxian Shen \thanks{Department of Mathematics and Statistics,
Auburn University,
 AL 36849, U. S. A. (wenxish@auburn.edu), partially supported by the NSF grant DMS--1645673.}}

\date{}

\maketitle

\begin{abstract}
The current series of research papers is to investigate the asymptotic dynamics in logistic type chemotaxis models in one space dimension with a free boundary or an unbounded boundary.
Such a model with a free boundary  describes the spreading of a new or invasive species subject to the influence of some chemical substances
in an environment with a free boundary representing the spreading front.
In this first part of the series,  we  investigate the dynamical behaviors of logistic type chemotaxis models on the half line $\mathbb{R}^+$, which are formally corresponding limit systems of the free boundary problems.
In the second of the series, we will establish the spreading-vanishing dichotomy in  chemoattraction-repulsion systems with a free boundary as well as with double free boundaries.
%In the third part of the series, we will  study the existence of  spreading speed and semi-wave solutions chemoattraction-repulsion systems %with a free boundary.
\end{abstract}


\medskip


\section{Introduction}

The current series of research papers is  to  study  spreading and vanishing dynamics of
 the following attraction-repulsion chemotaxis system with a free boundary and time and space dependent logistic source,
\begin{equation}\label{one-free-boundary-eq}
\begin{cases}
u_t = u_{xx} -\chi_1  (u  v_{1,x})_x + \chi_2 (u v_{2,x})_x + u(a(t,x) - b(t,x)u), \quad 0<x<h(t)
\\
 0 = \partial_{xx} v_1 - \lambda _1v_1 + \mu_1u,  \quad  0<x<h(t)
 \\
 0 = \partial_{xx}v_2 - \lambda_2 v_2 + \mu_2u,  \quad 0<x<h(t)
 \\
 h'(t) = -\nu u_x(t,h(t))
\\
u_x(t,0) = v_{1,x}(t,0) = v_{2,x}(t,0) = 0
\\
 u(t,h(t)) = v_{1,x}(t,h(t)) = v_{2,x}(t,h(t)) = 0
 \\
 h(0) = h_0,\quad u(x,0) = u_0(x),\quad  0\leq x\leq h_0,
\end{cases}
\end{equation}
and to study the
asymptotic dynamics of
\begin{equation}
\label{half-line-eq1}
\begin{cases}
u_t = u_{xx} -\chi_1  (u  v_{1,x})_x +\chi_2(u v_{2,x})_x+ u(a(t,x) - b(t,x)u),\quad x\in (0,\infty)
\cr
 0 = v_{1,xx} - \lambda_1v_1 + \mu_1u,  \quad x\in (0,\infty)\cr
 0=v_{2,xx}-\lambda_2 v_2+\mu_2 u,  \quad x\in (0,\infty)\cr
u_x(t,0)=v_{1,x}(t,0)=v_{2,x}(t,0)=0,
\end{cases}
\end{equation}
where $\nu>0$ in \eqref{one-free-boundary-eq} is a positive constant,  and in both \eqref{one-free-boundary-eq} and \eqref{half-line-eq1},   $\chi_i$, $\lambda_i$, and $\mu_i$ ($i=1,2$)  are nonnegative constants,   and $a(t,x)$ and $b(t,x)$  satisfy the following assumption,

\medskip

\noindent {\bf (H0)} {\it $a(t,x)$ and $b(t,x)$ are bounded $C^1$ functions on $\RR\times [0,\infty)$,
and
$$a_{\inf}:=\inf_{t\in\RR,x\in [0,\infty)}a(t,x)>0,\quad b_{\inf}:=\inf_{t\in\RR,x\in[0,\infty)}b(t,x)>0.
$$
}

Chemotaxis is the influence of chemical substances in the environment on the movement of mobile species.This can lead to strictly oriented movement or to partially oriented and partially tumbling movement. The movement towards a higher concentration of the chemical substance is termed positive chemotaxis and the movement towards regions of lower chemical concentration is called negative chemotaxis. The substances that lead to positive chemotaxis are chemoattractants and those leading to negative chemotaxis are so-called repellents.

 One of the first mathematical models of chemotaxis was introduced by Keller and Segel (\cite{Keller1970Initiation}, \cite{Keller1971Model}) to describe the aggregation of certain type of bacteria. A simplified version of their model involves the distribution $u$ of the density of the slime mold \textit{Dyctyostelum discoideum} and the concentration $v$ of a certain chemoattractant satisfying the following system of partial differential equations
\begin{equation}\label{KS}
\begin{cases}
 u_t = \nabla\cdot(\nabla u - \chi u\nabla v) + G(u), \quad x\in\Omega
 \\
 \epsilon v_t =d\Delta v + F(u,v),\quad x\in\Omega
\end{cases}
\end{equation}
complemented with certain boundary condition on $\partial \Omega$ if $\Omega$ is bounded, where $\Omega \subset \mathbb{R}^N$ is an open domain, $\epsilon \geq 0$ is a non-negative constant linked to the speed of diffusion of the chemical, $\chi$ represents the sensitivity with respect to chemotaxis, and the functions $G$ and $F$ model the growth of the mobile species and the chemoattractant, respectively.

Since their publication, considerable progress has been made in the analysis of various particular cases of \eqref{KS} on both bounded and unbounded fixed domains (see \cite{Bellomo2015Toward}, \cite{Diaz1995Symmtr}, \cite{Diaz1998Symmtr}, \cite{Galakhov2016on}, \cite{Horstmann2005bound}, \cite{Kanga2016blow}, \cite{Nagai1997application}, \cite{Sugiyama2006global1}, \cite{Sugiyama2006global2}, \cite{Wang2014on}, \cite{Winkler2010aggregation}, \cite{Winkler2011blow}, \cite{Winkler2013finite}, \cite{Winkler2014global}, \cite{Winkler2014how}, \cite{Yokota2015existence}, \cite{Zheng2015boundedness}, and the references therein). Among the central problems are
the existence of nonnegative solutions of \eqref{KS} which are globally defined in time or blow up at a finite time and the asymptotic behavior of time global solutions.
When $\epsilon >0 $ \eqref{KS} is referred to as the parabolic-parabolic Keller-Segel model and $\epsilon = 0$, which models the situation where the
chemoattractant diffuses  very quickly, is the case of parabolic-elliptic Keller-Segel model.  The reader is referred to \cite{HiPa, Horstmann2003from1970} for some detailed introduction into the mathematics of KS models.

When the cells undergo random motion and chemotaxis towards attractant and away from repellent \cite{luca2003chemotactic} on a fixed domain, we have a chemoattraction-repulsion process, which combined with proliferation and death of cells  leads to the following parabolic-elliptic-elliptic differential equations,
\begin{equation}\label{KS2}
\begin{cases}
 \frac{\partial u}{\partial t} =  \Delta u - \chi_1\nabla\cdot(u\nabla v_1) +  \chi_2\nabla\cdot(u\nabla v_2) + G(u), \quad x\in\Omega
 \\
 \epsilon \frac{\partial v_1}{\partial t} = d_1\Delta v_1 + F(u,v_1),\quad x\in\Omega
 \\
  \epsilon \frac{\partial v_2}{\partial t} = d_2\Delta v_2 + H(u,v_2),\quad x\in\Omega,
\end{cases}
\end{equation}
where $\chi_1, \chi_2$ are positive constants and system \eqref{KS2} becomes to \eqref{KS} automatically when $\chi_2 = 0.$ Compared to the studies of \eqref{KS}, the global existence of classical solutions on bounded or unbounded domain, and the stability of  equilibrium solutions of \eqref{KS2} are also studied in many papers (see \cite{espejo2014global, horstmann2011generalizing, jin2015boundedness, lin2016boundedness, liu2012classical, luca2003chemotactic, Salako2017Global, wang2016global, Wang2016bondedness, Zhang2016attraction, Zheng2016boundedness} and the references therein).

System \eqref{one-free-boundary-eq} describes the movement of a mobile species with population density $u(t,x)$ in an environment
with a free boundary  subject to
a chemoattractant with population density $v_1(t,x)$, which  diffuses very quickly,
 and a repellent with population density $v_2(t,x)$, which also diffuses very quickly.
  Due to the lack of first principles for the ecological situation under consideration, a thorough justification of the free boundary condition is difficult to supply. As in \cite{Bunting2012spreading}, we present in the following a derivation of the free boundary condition
  in \eqref{one-free-boundary-eq}
  based on the consideration of ``population loss" at the front and the assumption that, near the propagating front,
   population density is  close to zero.  Then,
    in the process of population range expansion, on one hand, the individuals of the species are suffering from
    the Allee effect near the propagating front. On the other hand, as the front enters new unpopulated environment, the pioneering members at the front, with very low population density, are particularly vulnerable.
     Therefore it is plausible to assume that as the expanding front propagates, the population suffers a loss of $\kappa$ units per unit volume at the front.

By Fick's first law, for a small time increment $\Delta t$, during the period from $t$ to $t +\Delta t$, the number of individuals of the population that enter  the region (through diffusion, or random walk) bounded by the old front $x = h(t)$ and new front $x = h(t+\Delta t)$ is approximated by $-d u_x(t,h(t))  \Delta t $
 (note that $u_x(t,h(t))\le 0$ for $u(t,x)\ge 0$ on $[0,h(t))$), where $d$ is some positive constant. The population loss in this region is approximated by
 $$
 \kappa \times (\mbox{volume of the region}) = \kappa \times [h(t + \Delta t) - h(t)].
 $$
 So the average density of the population in the region bounded by the two fronts is given by
 $$
 \frac{-d u_x (t,h(t)) \Delta t }{ h(t+\Delta t)-  h(t)} - \kappa.
 $$
 As $\Delta t \to 0$, the limit of this quantity is the population density at the front, namely $u(t,h(t))$, which by assumption is 0.
 This implies that
 $$
  h^{'}(t)=-\nu u_x(t,h(t))
$$
with $\nu=d/\kappa$, and the free boundary condition in  \eqref{one-free-boundary-eq} is then derived.

Consider \eqref{one-free-boundary-eq}, it is interesting to know  whether the species will spread into the whole region $[0,\infty)$ or will vanish eventually.
Formally, \eqref{half-line-eq1} can be viewed as the limit system of \eqref{one-free-boundary-eq} as $h(t)\to \infty$.
The study of the  asymptotic dynamics of \eqref{half-line-eq1} plays an important role in the characterization of  the spreading-vanishing dynamics
 of \eqref{one-free-boundary-eq} and is also of independent interest.
 The objective of this series is to investigate the asymptotic dynamics of \eqref{half-line-eq1} and
  the spreading and vanishing scenario in \eqref{one-free-boundary-eq}.

In this first part of the series, we investigate the asymptotic dynamics of \eqref{half-line-eq1} as well as the asymptotic dynamics
of the following chemotaxis system on the whole line,
\begin{equation}
\label{whole-line-eq1}
\begin{cases}
u_t = u_{xx} -\chi_1  (u  v_{1,x})_x +\chi_2(u v_{2,x})_x+ u(a(t,x) - b(t,x)u),\quad x\in \RR
\cr
 0 = v_{1,xx} - \lambda_1v_1 + \mu_1u,  \quad x\in \RR\cr
 0=v_{2,xx}-\lambda_2 v_2+\mu_2 u,  \quad x\in \RR.
\end{cases}
\end{equation}
Formally, \eqref{whole-line-eq1} can be viewed as the limit of the following free boundary problem with double free boundaries
\begin{eqnarray}\label{two-free-boundary-eq}
\begin{cases}
u_t = u_{xx} -\chi_1  (u  v_{1,x})_x  + \chi_2 (u v_{2,x})_x + u(a(t,x) - b(t,x) u),& x\in (g(t),h(t))
\\
 0 = (\partial_{xx} - \lambda_1I)v_1 + \mu_1u, & x\in  (g(t),h(t))
 \\
 0 = (\partial_{xx} - \lambda_2I)v_2 + \mu_2u, & x\in  (g(t),h(t))
 \\
g'(t) = -\nu u_x(g(t),t), h'(t) = -\nu u_x(h(t),t)
\\
u(g(t),t) = v_{1,x}(g(t),t) = v_{2,x}(g(t),t) = 0
\\
 u(h(t),t) = v_{1,x}(h(t),t) = v_{2,x}(h(t),t) = 0
\end{cases}
\end{eqnarray}
as $g(t)\to -\infty$ and $h(t)\to \infty$.
The investigation of the asymptotic dynamics of \eqref{whole-line-eq1}  then plays a role in
 the characterization of  the spreading-vanishing dynamics of  \eqref{two-free-boundary-eq}
 and is also of independent interest.

 In the second of the series, we will establish spreading and vanishing dichotomy scenario in \eqref{one-free-boundary-eq}
 and \eqref{two-free-boundary-eq}.
 %We will study the spreading speed and semi-wave solutions of \eqref{one-free-boundary-eq} in the third
 %part of the series.

 In the following, we state the main results of this paper.

Let
$$
C_{\rm unif}^b(\RR^+)=\{u\in C(\RR^+)\,|\, u(x)\,\, \text{is uniformly continuous and bounded on}\,\, \RR^+\}
$$
with norm $\|u\|_\infty=\sup_{x\in\RR^+}|u(x)|$, and
$$
C_{\rm unif}^b(\RR)=\{u\in C(\RR)\,|\, u(x)\,\, \text{is uniformly continuous and bounded on}\,\, \RR\}
$$
with norm $\|u\|_\infty=\sup_{x\in\RR}|u(x)|$. Define
\begin{align}\label{m-eq}
M= \min\Big\{ &\frac{1}{\lambda_2}\big( (\chi_2\mu_2\lambda_2-\chi_1\mu_1\lambda_1)_+ + \chi_1\mu_1(\lambda_1-\lambda_2)_+ \big),\nonumber\\
& \qquad \frac{1}{\lambda_1}\big( (\chi_2\mu_2\lambda_2-\chi_1\mu_1\lambda_1)_+ + \chi_2\mu_2(\lambda_1-\lambda_2)_+ \big) \Big\}
 \end{align}
 and
 \begin{align}\label{k-eq}
K=\min\Big\{&\frac{1}{\lambda_2}\Big(|\chi_1\mu_1\lambda_1-\chi_2\mu_2\lambda_2|+\chi_1\mu_1|\lambda_1-\lambda_2|\Big),\nonumber\\
&\quad  \frac{1}{\lambda_1}\Big(|\chi_1\mu_1\lambda_1-\chi_2\mu_2\lambda_2|+\chi_2\mu_2|\lambda_1-\lambda_2|\Big) \Big\}.
\end{align}
 Let {\bf (H1)}- {\bf (H3)}  be the following standing assumptions.

\medskip

\noindent {\bf (H1)}  $b_{\inf}>\chi_1\mu_1-\chi_2\mu_2+M$.

\medskip

\noindent {\bf (H2)} $b_{\inf}>\Big(1+\frac{a_{\sup}}{a_{\inf}}\Big)\chi_1\mu_1-\chi_2\mu_2+M$.

\medskip

\noindent {\bf (H3)}  $b_{\inf}>\chi_1\mu_1-\chi_2\mu_2+K$.

\medskip

The main results of this first part are stated in the following theorems.

\begin{thm}[Global existence]
\label{half-line-thm1}
Consider \eqref{half-line-eq1}.
If {\bf (H1)} holds, then for any $t_0\in\RR$ and any nonnegative function $u_0\in C^{b}_{\rm unif}(\RR^+)$, \eqref{half-line-eq1} has a unique solution
    $(u(t,x;t_0,u_0),v_1(t,x;t_0,u_0)$, $v_2(t,x;t_0,u_0))$ with $u(t_0,x;t_0,u_0)=u_0(x)$ defined for $t\ge t_0$. Moreover,
    $$
0\le u(t,x;t_0,u_0)\le  C(u_0)\quad \forall \,\, t\in [t_0,\infty), \,\, x\in [0,\infty),
$$
{ and
$$
\limsup_{t\to\infty}\|u(t,\cdot;t_0,u_0)\|_\infty\le \frac{a_{\sup}}{b_{\inf}+\chi_2\mu_2-\chi_1\mu_1-M},
$$}
where
\begin{equation}
\label{c0-eq}
C(u_0)=\max\{\|u_0\|_{\infty}, \frac{a_{\sup}}{b_{\inf}+\chi_2\mu_2-\chi_1\mu_1-M}\}.
\end{equation}
\end{thm}

\begin{thm}[Persistence]
\label{half-line-thm2}

Consider \eqref{half-line-eq1}.

 \begin{itemize}

 \item[(1)]
 If {\bf (H1)} holds, then for any $u_0\in C_{\rm unif}^b(\RR^+)$ with $\inf_{x\in\RR^+}u_0(x)>0$, there is $m(u_0)>0$ such that
 $$
 m(u_0)\le u(t,x;t_0,u_0)\le C(u_0)\quad \forall\,\, t\ge t_0,\,\, x\in\RR^+.
 $$

 \item[(2)] If {\bf (H2)} holds, then there are $0<m_0<M_0$ such that for any $u_0\in C_{\rm unif}^b(\RR^+)$
 { with $\inf_{x\in\RR^+}u_0(x)>0$}, there is $T(u_0)>0$ such that
 $$
 m_0\le u(t,x;t_0,u_0)\le M_0\quad \forall \,\,{  t_0\in\RR},\,\, t\ge t_0+T(u_0),\,\,  x\in\RR^+.
 $$
\end{itemize}

\end{thm}

\begin{thm}[Positive entire solution]
\label{half-line-thm3}

Consider \eqref{half-line-eq1}.

\begin{itemize}

\item[(1)] (Existence of strictly positive entire solution) If {\bf (H1)} holds, then \eqref{half-line-eq1}
admits a strictly positive entire solution $(u^+(t,x),v_1^+(t,x),v_2^+(t,x))$. Moreover, if
 $a(t+T,x)\equiv a(t,x)$ and $b(t+T,x)\equiv b(t,x)$, then \eqref{half-line-eq1}
admits a strictly positive $T-$ periodic solution $ (u^+(t,x)$, $v_1^+(t,x)$, $v_2^+(t,x))= (u^+(t+T,x),v_1^+(t+T,x),v_2^+(t+T,x)) $.

\item[(2)] (Stability and uniqueness of strictly positive entire solution)

\begin{itemize}

\item[(i)]  Assume {\bf (H3)} and $a(t,x)\equiv a(t)$ and $b(t,x)\equiv b(t)$, then  for any $u_0\in C_{\rm unif}^b(\RR^+)$ with $\inf_{x\in \RR}u_0(x)>0$,
   \begin{equation}\label{Eq-asymptotic-0}
 \lim_{t\to\infty} \| u(t+t_0,\cdot;t_0,u_0) - u^+(t+t_0,\cdot)\|_\infty =0, \forall t_0\in \mathbb{R}.
\end{equation}

\item[(ii)] Assume {\bf (H3)}.
There are $\chi_1^*>0$ and $\chi_2^*>0$ such that, if $0\le \chi_1\le \chi_1^*$ and $0\le \chi_2\le \chi_2^*$, then for any $u_0\in C_{\rm unif}^b(\RR^+)$ with $\inf_{x\in \RR^+}u_0(x)>0$,
\eqref{Eq-asymptotic-0} holds.
\end{itemize}
\end{itemize}
\end{thm}

 Similar results to Theorems \ref{half-line-thm1}-\ref{half-line-thm3}
  hold for \eqref{whole-line-eq1}. More precisely, we have

\begin{thm}
\label{whole-line-thm}
Consider \eqref{whole-line-eq1}. The following hold.
\begin{itemize}
\item[(1)] (Global existence) If {\bf (H1)} holds, then for any $t_0\in\RR$ and any nonnegative function $u_0\in C^{b}_{\rm unif}(\RR)$, \eqref{whole-line-eq1} has a unique solution
    $(u(t,x;t_0,u_0),v_1(t,x;t_0,u_0)$, $v_2(t,x;t_0,u_0))$ with $u(t_0,x;t_0,u_0)=u_0(x)$ defined for $t\ge t_0$. Moreover,
    $$
0\le u(t,x;t_0,u_0)\le  C(u_0)\quad \forall \,\, t\in [t_0,\infty), \,\, x\in \RR,
$$
{ and
$$
\limsup_{t\to\infty}\|u(t,\cdot;t_0,u_0)\|_\infty\le \frac{a_{\sup}}{b_{\inf}+\chi_2\mu_2-\chi_1\mu_1-M},
$$}
where
\begin{equation}
\label{c0-eq}
C(u_0)=\max\{\|u_0\|_{\infty}, \frac{a_{\sup}}{b_{\inf}+\chi_2\mu_2-\chi_1\mu_1-M}\}.
\end{equation}

\item[(2)] (Persistence)

 \begin{itemize}

 \item[(i)]
 If {\bf (H1)} holds, then for any $u_0\in C_{\rm unif}^b({ \RR})$ with $\inf_{{ x\in\RR}}u_0(x)>0$, there is $m(u_0)>0$ such that
 $$
 m(u_0)\le u(t,x;t_0,u_0)\le C(u_0)\quad \forall\,\, t\ge t_0,\,\, x\in{ \RR}.
 $$

 \item[(ii)] If {\bf (H2)} holds, then there are $0<m_0<M_0$ such that for any $u_0\in C_{\rm unif}^b({ \RR})$ with
 $\inf_{x\in\RR}u_0(x)>0$, there is $T(u_0)>0$ such that
 $$
 m_0\le u(t,x;t_0,u_0)\le M_0\quad \forall t\ge t_0+T(u_0),\quad x\in { \RR}.
 $$
\end{itemize}

\item[(3)] (Existence of strictly positive entire solution) If {\bf (H1)} holds, then \eqref{whole-line-eq1}
admits a strictly positive entire solution $(u^+(t,x),v_1^+(t,x),v_2^+(t,x))$. Moreover, if
 $a(t+T,x)\equiv a(t,x)$ and $b(t+T,x)\equiv b(t,x)$, then \eqref{whole-line-eq1}
admits a strictly positive $T-$ periodic solution $ (u^+(t,x)$, $v_1^+(t,x)$, $v_2^+(t,x))= (u^+(t+T,x),v_1^+(t+T,x),v_2^+(t+T,x)) $.

\item[(4)] (Stability and uniqueness of strictly positive entire solution)

\begin{itemize}

\item[(i)]  Assume {\bf (H3)} and $a(t,x)\equiv a(t)$ and $b(t,x)\equiv b(t)$, then  for any $u_0\in C_{\rm unif}^b(\RR)$ with $\inf_{x\in \RR}u_0(x)>0$,
   \begin{equation}\label{Eq-asymptotic-1}
 \lim_{t\to\infty} \| u(t+t_0,\cdot;t_0,u_0) - u^+(t+t_0,\cdot)\|_\infty =0, \forall t_0\in \mathbb{R}.
\end{equation}

\item[(ii)] Assume {\bf (H3)}.
There are $\chi_1^*>0$ and $\chi_2^*>0$ such that, if $0\le \chi_1\le \chi_1^*$ and $0\le \chi_2\le \chi_2^*$, then for any $u_0\in C_{\rm unif}^b(\RR)$ with $\inf_{x\in \RR}u_0(x)>0$,
\eqref{Eq-asymptotic-1} holds.
\end{itemize}
\end{itemize}

\end{thm}

\begin{rk}
\begin{itemize}
\item[(1)] Note  that $b_{\inf}\ge \chi_1\mu_1$ implies {\bf (H1)},  $b_{\inf}\ge \big(1+\frac{a_{\sup}}{a_{\inf}}\big)\chi_1\mu_1$ implies {\bf (H2)}, and $b_{\inf}>2\chi_1\mu_1$ implies {\bf (H3)}. In the case $\chi_2=0$, we can choose $\lambda_2=\lambda_1$, and {\bf (H1)} becomes
$b_{\inf}>\chi_1\mu_1$, {\bf (H2)} becomes   $b_{\inf}\ge \big(1+\frac{a_{\sup}}{a_{\inf}}\big)\chi_1\mu_1$,  and {\bf (H3)} becomes $b_{\inf}>2\chi_1\mu_1$.

\item[(2)]  In \cite{Salako2017Global}, an attraction-repulsion chemotaxis system with
constant logistic source $u(a-bu)$  on the whole space is studied. Among others, it is proved in \cite{Salako2017Global}
that if  {\bf (H1)} holds, then \eqref{whole-line-eq1} with $a(t,x)\equiv a$ and $b(t,x)\equiv b$ has a unique globally defined
 solution for any nonnegative, bounded, and uniformly continuous initial function (see \cite[Theorem A]{Salako2017Global}), and that if
 {\bf (H3)} holds, then the constant solution $(\frac{a}{b},\frac{\mu_1}{\lambda_1}\frac{a}{b}, \frac{\mu_2}{\lambda_2}\frac{a}{b})$
 is globally stable with strictly positive perturbations (see \cite[Theorem B]{Salako2017Global}).
 Theorem \ref{whole-line-thm} extends  \cite[Theorem A]{Salako2017Global} and \cite[Theorem B]{Salako2017Global} for \eqref{whole-line-eq1} with constant logistic source to time and space dependent logistic source. It should be mentioned that  in \cite{Zhang2016attraction}, an attraction-repulsion chemotaxis system with
constant logistic source $u(a-bu)$  on a bounded domain with Neumann boundary conditions is studied.

 \item[(3)] \eqref{whole-line-eq1} with $\chi_2=0$ is a special cases of the parabolic-elliptic chemotaxis model with space-time dependent logistic sources on $\RR^N$ studied in \cite{Salako2018parabolic1} and \cite{Salako2018parabolic2}. Theorem \ref{whole-line-thm} in the case
     $\chi_2=0$ is proved in  \cite{Salako2018parabolic1} and \cite{Salako2018parabolic2} (see \cite[Theorem 1.1]{Salako2018parabolic1},
     \cite[Theorem 1.4]{Salako2018parabolic2}, and \cite[Theorem 1.5]{Salako2018parabolic2}). Theorem \ref{whole-line-thm} also extends
      \cite[Theorem 1.1]{Salako2018parabolic1},
     \cite[Theorem 1.4]{Salako2018parabolic2}, and \cite[Theorem 1.5]{Salako2018parabolic2}) for the parabolic-elliptic chemotaxis model with space-time dependent logistic sources on the whole space to the parabolic-elliptic-elliptic chemotaxis model with space-time dependent logistic sources on the whole space.

 \item[(4)] Logistic type attraction-repulsion chemotaxis systems on a half space are studied for the first time.
 The results stated in Theorems \ref{half-line-thm1}-\ref{half-line-thm3} are similar to those stated in Theorem \ref{whole-line-thm} for logistic type attraction-repulsion chemotaxis systems on the whole space. Several  existing techniques developed for the  study of
 logistic type attraction-repulsion chemotaxis systems on a whole space are applied for the study of \eqref{half-line-eq1}
 with certain modifications.   But, due to the presence of the boundary $x=0$ as well as the unboundedness of the domain,
 such modifications are nontrivial and  some other technical difficulties also arise in the study of \eqref{half-line-eq1}.
 \end{itemize}
\end{rk}

The rest of this paper is organized in the following way. In section 2, we present some preliminary lemmas to be used in the proofs
of the main results.
We prove the main results of the paper in section 3.

\section{Preliminary lemmas}

In this section, we present some lemmas to be used in the proof of the main results in later sections.

The first lemma is on the local existence of solutions of \eqref{half-line-eq1} and \eqref{whole-line-eq1}.

\begin{lem}
\label{half-line-lm1}
\begin{itemize}
\item[(1)] Consider \eqref{whole-line-eq1}. For any $t_0\in\RR$ and any nonnegative function $u_0\in C^{b}_{\rm unif}(\RR)$, there is $T_{\max}>0$ such that \eqref{whole-line-eq1} has a unique solution
    $(u(t,x;t_0,u_0),v_1(t,x;t_0,u_0)$, $v_2(t,x;t_0,u_0))$ defined on $[t_0, t_0+T_{\max})$ with $u(t_0,x;t_0,u_0)=u_0(x)$.
    Moreover, if $T_{\max}<\infty$, then
    $$
    \limsup_{t\to T_{\max}} \|u(t_0+t,\cdot;t_0,u_0)\|_\infty=\infty.
    $$

\item[(2)] Consider \eqref{half-line-eq1}.
For any $t_0\in\RR$ and any nonnegative function $u_0\in C^{b}_{\rm unif}(\RR^+)$, there is $T_{\max}>0$ such that \eqref{half-line-eq1} has a unique solution
    $(u(t,x;t_0,u_0),v_1(t,x;t_0,u_0)$, $v_2(t,x;t_0,u_0))$ defined on $[t_0, t_0+T_{\max})$ with $u(t_0,x;t_0,u_0)=u_0(x)$.
    Moreover, if $T_{\max}<\infty$, then
    $$
    \limsup_{t\to T_{\max}} \|u(t_0+t,\cdot;t_0,u_0)\|_\infty=\infty.
    $$
    \end{itemize}
\end{lem}

\begin{proof}
(1) It follows from the similar arguments used in the proof of \cite[Theorem 1.1]{Salako2017GlobalE}. For the reader's convenience and
for the proof of (2), we outline the proof in the following.

First, let $T(t)$ be the semigroup generated by $\p_{xx}- I$ on $C_{\rm unif}^b(\RR)$. Then for any $u_0\in\ C_{\rm unif}^b(\RR)$,
\begin{align*}
(T(t)u_0)(x)=\frac{e^{-t}}{2\sqrt \pi}\int_{-\infty}^\infty \frac{1}{\sqrt t} e^{-\frac{(x-y)^2}{4t}}  u_0(y)dy
\end{align*}
for $t>0$ and $x\in\RR$.
Let  $u\in C_{\rm unif}^b(\RR)$ and set  $v=(\p_{xx}-\lambda I)^{-1}u.$  Then we have
 \begin{equation}
 \label{asym-eq0}
v(x)=\frac{1}{2\sqrt \pi}\int_{0}^{\infty}\int_{-\infty}^\infty \frac{e^{-\lambda s}}{\sqrt s}e^{-\frac{|x-z|^{2}}{4s}} u(z)dzds.
\end{equation}
 By \cite[Lemma 3.2]{Salako2017GlobalE},
$T(t)\p_x$ can be extended to $C_{\rm unif}^b(\RR)$,  and
for any $u\in C_{\rm unif}^b(\RR)$, there holds
$$
\|\big(T(t)\p_x\big)u\|_\infty\le \frac{1}{\sqrt \pi}t^{-\frac{1}{2}}\|u\|_\infty.
$$
By \cite[Lema 3.3]{Salako2017GlobalE}, for any $u\in C_{\rm unif}^b(\RR)$,
$$
\|\p_x(\p_{xx}-\lambda I)^{-1}u\|_\infty\le \frac{1}{\sqrt \lambda} \|u\|_\infty.
$$

By the similar arguments as those in \cite[Theorem 1.1]{Salako2017GlobalE},  there is $\tau>0$  such that \eqref{whole-line-eq1} has a unique solution
    $(u(t,x),v_1(t,x),v_2(t,x))=(u(t,x;t_0,u_0),v_1(t,x;t_0,u_0)$, $v_2(t,x;t_0,u_0))$ with $u(t_0,x;t_0,u_0)=u_0(x)$ defined on $[t_0, t_0+\tau)$
    and  satisfying
 \begin{align*}
u(t, \cdot)=& T(t-t_0)u_{0} +\chi_1\int_{t_0}^{t} (T(t-s)\p_x) ( u(s,\cdot) \p_x (\p_{xx}-\lambda_1 I)^{-1}u(s))ds\\
& -\chi_2\int_{t_0}^{t} (T(t-s)\p_x) ( u(s,\cdot) \p_x(\p_{xx}-\lambda_2 I)^{-1}u(s))ds\\
& +\int_{t_0}^{t} T(t-s)\big(1+a(s,\cdot)\big)u(s,\cdot)ds-\int_{t_0}^{t} T(t-s) b(s,\cdot)u^{2}(s,\cdot)ds.
\end{align*}

Now, by the standard extension arguments, there is $T_{\max}>0$ such that  \eqref{whole-line-eq1} has a unique solution
    $(u(t,x;t_0,u_0),v_1(t,x;t_0,u_0)$, $v_2(t,x;t_0,u_0))$ with $u(t_0,x;t_0,u_0)=u_0(x)$ defined on $[t_0, t_0+T_{\rm max})$, and
    if $T_{\rm max}<\infty$, then
    $$
    \limsup_{t\to t_0+T_{\rm max}} \|u(t,\cdot;t_0,u_0)\|_\infty=\infty.
    $$

(2) It can be proved by the arguments in (1). To be more precise, first,  let $\tilde T(t)$ be the semigroup generated by $\p_{xx}- I$ on $C_{\rm unif}^b(\RR^+)$ with Neumann boundary at $0$. Then for any $u_0\in\ C_{\rm unif}^b(\RR^+)$,
\begin{align*}
(\tilde T(t)u_0)(x)&=\frac{e^{-t}}{2\sqrt \pi}\int_0^\infty \frac{1}{\sqrt t}\Big[ e^{-\frac{(x-y)^2}{4t}}+e^{-\frac{(x+y)^2}{4t}}\Big] u_0(y)dy\\
&=\frac{e^{-t}}{2\sqrt \pi}\int_{-\infty}^\infty \frac{1}{\sqrt t} e^{-\frac{(x-y)^2}{4t}}\tilde  u_0(y)dy\\
&=T(t)\tilde u_0
\end{align*}
for $t>0$ and $x\in\RR^+$, where $\tilde u_0(x)=u_0(|x|)$ for $x\in\RR$.
Let  $u\in C_{\rm unif}^b(\RR^+)$ and set  $v=(\p_{xx}-\lambda I)^{-1}u $ on $[0,\infty)$.  Then we have
 \begin{equation}
 \label{asym-eq0}
v(x)=\frac{1}{2\sqrt \pi}\int_{0}^{\infty}\int_{-\infty}^\infty \frac{e^{-\lambda s}}{\sqrt s}e^{-\frac{|x-z|^{2}}{4s}}\tilde u(z)dzds
\end{equation}
for every $x\in\RR^+$, where $\tilde u(z)=u(|z|)$. Hence by the arguments in (1),
$\tilde T(t)\p_x$ can be extended to $C_{\rm unif}^b(\RR^+)$,  and
for any $u\in C_{\rm unif}^b(\RR^+)$, there holds
$$
\|\big(\tilde T(t)\p_x\big)u\|_\infty\le \frac{1}{\sqrt \pi}t^{-\frac{1}{2}}\|u\|_\infty.
$$
Also,  for any $u\in C_{\rm unif}^b(\RR^+)$,
$$
\|\p_x(\p_{xx}-\lambda I)^{-1}u\|_\infty\le \frac{1}{\sqrt \lambda} \|u\|_\infty.
$$
We can then apply the arguments in \cite[Theorem 1.1]{Salako2017GlobalE} to prove that there is $\tau>0$  such that \eqref{half-line-eq1} has a unique solution
    $(u(t,x),v_1(t,x),v_2(t,x))=(u(t,x;t_0,u_0),v_1(t,x;t_0,u_0)$, $v_2(t,x;t_0,u_0))$ with $u(t_0,x;t_0,u_0)=u_0(x)$ defined on $[t_0, t_0+\tau)$
    and  satisfying
 \begin{align*}
u(t, \cdot)=& \tilde T(t-t_0)u_{0} +\chi_1\int_{t_0}^{t} (\tilde T(t-s)\p_x) ( u(s,\cdot) \p_x (\p_{xx}-\lambda_1 I)^{-1}u(s))ds\\
& -\chi_2\int_{t_0}^{t} (\tilde T(t-s)\p_x) ( u(s,\cdot) \p_x(\p_{xx}-\lambda_2 I)^{-1}u(s))ds\\
& +\int_{t_0}^{t}\tilde  T(t-s)\big(1+a(s,\cdot)\big)u(s,\cdot)ds-\int_{t_0}^{t} \tilde T(t-s) b(s,\cdot)u^{2}(s,\cdot)ds.
\end{align*}

Next, by the standard extension arguments, there is $T_{\max}>0$ such that  \eqref{half-line-eq1} has a unique solution
    $(u(t,x;t_0,u_0),v_1(t,x;t_0,u_0)$, $v_2(t,x;t_0,u_0))$ with $u(t_0,x;t_0,u_0)=u_0(x)$ defined on $[t_0, t_0+T_{\rm max})$, and
    if $T_{\rm max}<\infty$, then
    $$
    \limsup_{t\to t_0+T_{\rm max}} \|u(t,\cdot;t_0,u_0)\|_\infty=\infty.
    $$
\end{proof}

The second lemma is on the estimate of  $\chi_2 \lambda_2 v_2-\chi_1 \lambda_1 v_1$.

\begin{lem}
\label{half-line-lm2}
Assume {\bf (H1)} holds.
\begin{itemize}
\item[(1)]
Suppose that $(u(t,x;t_0,u_0),v_1(t,x;t_0,u_0),v_2(t,x;t_0,u_0))$ is a solution of \eqref{half-line-eq1} (resp. \eqref{whole-line-eq1}) on $[t_0,t_0+T]$ with $u(t,\cdot;t_0,u_0)=u_0(\cdot)$, where $u_0\in C_{\rm unif}^b(\RR^+)$ (resp.   $u_0\in C_{\rm unif}^b(\RR)$).
If $\|u(t,\cdot;t_0,u_0)\|_\infty\le C_0:=C(u_0)$ for $t\in [t_0,t_0+T]$, then
$$
(\chi_2 \lambda_2 v_2-\chi_1 \lambda_1 v_1)(t,x;t_0,u_0)\leq M C_{0}\quad \forall\, t\in [t_0,t_0+T],
$$
where $M$ and $C_0(u_0)$ are as in \eqref{m-eq} and \eqref{c0-eq}, respectively.

\item[(2)] Let $u(\cdot,\cdot)\in C_{\rm unif}^b([t_0,\infty)\times \RR^+)$ and $v_1(t,x;u)$ and $v_2(t,x;u)$ be the solutions of
\begin{equation*}
\begin{cases}
v_{1,xx}-\lambda_1v_1+\mu_1 u=0,\quad x\in (0,\infty)\cr
v_{1,x}(t,0)=0
\end{cases}
\end{equation*}
and
\begin{equation*}
\begin{cases}
v_{2,xx}-\lambda_2v_1+\mu_2 u=0,\quad x\in (0,\infty)\cr
v_{2,x}(t,0)=0,
\end{cases}
\end{equation*}
respectively.
Then
$$
|\chi_2 \mu_2 v_2(t,\cdot;u)-\chi_1 \mu_1 v_1(t,\cdot;u)\|_\infty \leq K \|u(t,\cdot)\|_\infty \quad \forall\, t\in [t_0,\infty),
$$
where $K$ is as in \eqref{k-eq}.
\end{itemize}

\end{lem}

\begin{proof}
(1) It can be proved by some similar arguments as those in \cite[Theorem A]{Salako2017Global}. For the completeness,
we provide a proof  for \eqref{half-line-eq1}. It can be proved similarly for \eqref{whole-line-eq1}.

Let $(u(t,x),v_1(t,x),v_2(t,x))=(u(t,x;t_0,u_0),v_1(t,x;t_0,u_0),v_2(t,x;t_0,u_0))$.
Note that $v_1(t,x)$ is the solution of
$$
\begin{cases}
v_{1,xx}-\lambda_1 v_1+\mu_1 u(t,x)=0,\quad 0<x<\infty\cr
v_{1,x}(t,0)=0
\end{cases}
$$
and $v_2(t,x)$ is the solution of
$$
\begin{cases}
v_{2,xx}-\lambda_2 v_2+\mu_2 u(t,x)=0,\quad 0<x<\infty\cr
v_{2,x}(t,0)=0.
\end{cases}
$$
Let $\bar T(s)$ be the semigroup generated by $\partial_{xx}$ on $C_{\rm unif}^b(\RR^+)$  with Neumann boundary at $x=0$.
Then
$$
v_i(t,\cdot)=\mu_i \int_0^ \infty e^{-\lambda_i s} \bar T(s) u(t,\cdot)ds,\quad i=1,2.
$$
We then have
\begin{align*}
(\chi_2 \lambda_2 v_2-\chi_1 \lambda_1 v_1)(t,x)&=\chi_2 \lambda_2 \mu_2 \int_0^ \infty e^{-\lambda_2 s} \bar T(s) u(t,\cdot)ds-
\chi_1 \lambda_1 \mu_1 \int_0^ \infty e^{-\lambda_1 s} \bar T(s) u(t,\cdot)ds\\
&=\big(\chi_2 \lambda_2 \mu_2-\chi_1\lambda_1\mu_1\big) \int_0^ \infty e^{-\lambda_2 s} \bar T(s) u(t,\cdot)ds\\
&\,\, \,\,  +
\chi_1 \lambda_1 \mu_1 \int_0^ \infty \big(e^{-\lambda_2 s}-e^{-\lambda_1 s}\big) \bar T(s) u(t,\cdot)ds\\
& \le \big(\chi_2 \lambda_2 \mu_2-\chi_1\lambda_1\mu_1\big)_+ \int_0^ \infty e^{-\lambda_2 s}\bar  T(s) u(t,\cdot)ds\\
&\,\, \,\,  +
\chi_1 \lambda_1 \mu_1 \int_0^ \infty \big(e^{-\lambda_2 s}-e^{-\lambda_1 s}\big)_+ \bar T(s) u(t,\cdot)ds
\end{align*}
Note that
$$
\bar T(s)u(t,\cdot)\le \bar T(s) C_0=C_0.
$$
Hence
\begin{align*}
(\chi_2 \lambda_2 v_2-\chi_1 \lambda_1 v_1)(t,x)&\le  \big(\chi_2 \lambda_2 \mu_2-\chi_1\lambda_1\mu_1\big)_+ \int_0^ \infty e^{-\lambda_2 s} C_0 ds\\
&\,\,\,\,  +
\chi_1 \lambda_1 \mu_1 \int_0^ \infty \big(e^{-\lambda_2 s}-e^{-\lambda_1 s}\big)_+ C_0ds\\
&= \frac{C_0}{\lambda_2}\Big((\chi_2\lambda_2\mu_2-\chi_1\lambda_1\mu_1)_{+}+\chi_1\mu_1(\lambda_1-\lambda_2)_{+}\Big).
\end{align*}

Similarly, we can prove that
$$
 (\chi_2 \lambda_2 v_2-\chi_1 \lambda_1 v_1)(t,x)\le \frac{C_0}{\lambda_{1}}\Big( \chi_2\mu_2(\lambda_1-\lambda_2)_{+} + (\chi_2\mu_2\lambda_2-\chi_1\mu_1\lambda_1)_{+} \Big).
$$
(1) then follows.

(2) Note that
$$
v_i(t,\cdot;u)=\lambda _i \int_0^ \infty e^{-\lambda_i s} \bar T(s) u(t,\cdot)ds,\quad i=1,2.
$$
Hence
\begin{align*}
|(\chi_2 \mu_2 v_2-\chi_1 \mu_1 v_1)(t,x;u)|&=|\chi_2 \lambda_2 \mu_2 \int_0^ \infty e^{-\lambda_2 s} \bar T(s) u(t,\cdot)ds-
\chi_1 \lambda_1 \mu_1 \int_0^ \infty e^{-\lambda_1 s} \bar T(s) u(t,\cdot)ds|\\
&\le |\big(\chi_2 \lambda_2 \mu_2-\chi_1\lambda_1\mu_1\big) \int_0^ \infty e^{-\lambda_2 s} \bar T(s) u(t,\cdot)ds|\\
&\,\, \,\,  +
|\chi_1 \lambda_1 \mu_1 \int_0^ \infty \big(e^{-\lambda_2 s}-e^{-\lambda_1 s}\big) \bar T(s) u(t,\cdot)ds|\\
& \le |\chi_2 \lambda_2 \mu_2-\chi_1\lambda_1\mu_1| \int_0^ \infty e^{-\lambda_2 s}\bar  T(s) \|u(t,\cdot)\|_\infty ds\\
&\,\, \,\,  +
\chi_1 \lambda_1 \mu_1 \int_0^ \infty |e^{-\lambda_2 s}-e^{-\lambda_1 s}| \bar T(s) \|u(t,\cdot)\|_\infty ds\\
&= \Big(\frac{|\chi_2 \lambda_2 \mu_2-\chi_1\lambda_1\mu_1|}{\lambda_2}+\frac{|\lambda_2-\lambda_1|}{\lambda_2}\Big)  \|u(t,\cdot)\|_\infty
\end{align*}
for $t\in [t_0,\infty)$.
Similarly, it can be proved that
 \begin{align*}
|(\chi_2 \mu_2 v_2-\chi_1 \mu_1 v_1)(t,x;u)|\le  \Big(\frac{|\chi_2 \lambda_2 \mu_2-\chi_1\lambda_1\mu_1|}{\lambda_1}+\frac{|\lambda_2-\lambda_1|}{\lambda_1}\Big)  \|u(t,\cdot)\|_\infty
\end{align*}
for $t\in [t_0,\infty)$.  It then follows that
$$
|(\chi_2 \mu_2 v_2-\chi_1 \mu_1 v_1)(t,x;u)|\le  K \|u(t,\cdot)\|_\infty
$$
for $t\in [t_0,\infty)$. Hence (2) holds.
\end{proof}

The next lemma is on the upper bound of $u(t,x;t_0,u_0)$.

\begin{lem}
\label{half-line-lm3}
Consider \eqref{half-line-eq1} and assume {\bf (H1)}.
 For any given $t_0\in\RR$ and $u_0\in C_{\rm unif}^b(\RR^+)$ with $u_0\ge 0$ and $u_0\not =0$,  if $u(t,x;t_0,u_0)$ exists on $[t_0,\infty)$ and
 $\limsup_{t\to\infty} \|u(t,\cdot;t_0,u_0)\|<\infty$, then
$$
\limsup_{t\to\infty}\|u(t,\cdot;t_0,u_0)\|_\infty\le \frac{a_{\sup}}{b_{\inf}+\chi_2\mu_2-\chi_1\mu_1-M}.
$$
\end{lem}

\begin{proof}
(1) For given $t_0\in\RR$ and $u_0\in C^{b}_{\rm unif}(\RR^+)$ with $u_0\geq 0 $ and $u_0\not =0$,  assume that $u(t,x;t_0,u_0)$ exists on $[t_0,\infty)$
and  $\limsup_{t\to\infty} \|u(t,\cdot;t_0,u_0)<\infty$.
Let
$$
%\underline{u}=\liminf_{t\to\infty}\inf_{x\in\R^N}u(x,t;t_0,u_0)\quad \text{and}\quad
\overline{u}=\limsup_{t\to\infty}\sup_{x\in\R^N}u(x,t;t_0,u_0).
$$
By the assumption,  $\bar u<\infty$.
Then  for every $\varepsilon>0$, there is $T_{\varepsilon}>0$ such that
$$
 u(t+t_0,x;t_0,u_0)\leq \overline{u}+\varepsilon\quad \forall\ x\in\RR^+,\ \forall\ t\geq T_{\varepsilon}.
$$
Hence, it follows from comparison principle for  elliptic equations, that
\begin{equation}\label{aux-asym-eq1}
\lambda_iv_i(t+t_0,x;t_0,u_0)\leq \mu_i(\overline{u}+\varepsilon), \forall\ x\in\RR^+,\ \ \forall \ t\ge T_{\varepsilon},\ i=1,2.
\end{equation}

By similar arguments as those in Lemma \ref{half-line-lm2}, we have
\begin{align*}
(\chi_2 \lambda_2 v_2-\chi_1 \lambda_1 v_1)(t,x;t_0,u_0)\le  \frac{\bar u+\varepsilon}{\lambda_2}\Big((\chi_2\lambda_2\mu_2-\chi_1\lambda_1\mu_1)_{+}+\chi_1\mu_1(\lambda_1-\lambda_2)_{+}\Big)
\end{align*}
and
$$
 (\chi_2 \lambda_2 v_2-\chi_1 \lambda_1 v_1)(t,x;t_0,u_0)\le \frac{\bar u+\varepsilon}{\lambda_{1}}\Big( \chi_2\mu_2(\lambda_1-\lambda_2)_{+} + (\chi_2\mu_2\lambda_2-\chi_1\mu_1\lambda_1)_{+} \Big)
$$
for $t\ge t_0+T_\varepsilon$.
This implies that
$$
(\chi_2 \lambda_2 v_2-\chi_1 \lambda_1 v_1)(t,x;t_0,u_0)\le M (\bar u+\varepsilon)
$$
for $t\ge t_0+T_\varepsilon$ and then
\begin{eqnarray}
\label{aux-asym-eq1-1}
u_t\le  u_{xx}+(\chi_2v_2-\chi_1v_1)_x  u_x + (a_{\sup}+M (\overline{u}+\varepsilon) )u- (b_{\inf}+\chi_2\mu_2-\chi_1\mu_1)u)u
\end{eqnarray}
for $t\ge t_0+T_\varepsilon$.

By \eqref{aux-asym-eq1-1} and comparison principle for parabolic equations,
$$
u(t,x;t_0,u_0)\le U_\varepsilon(t)\quad \forall\,\, t\ge t_0+T_\varepsilon,\,\, x\in\RR^+,
$$
where $U_\varepsilon(t)$ is the solution of
$$
\begin{cases}
U^{'}= (a_{\sup}+M (\overline{u}+\varepsilon) )U- (b_{\inf}+\chi_2\mu_2-\chi_1\mu_1)U)U\cr
U(t_0+T_\varepsilon)=\|u(t_0+T_\varepsilon,\cdot;t_0,u_0)\|_\infty.
\end{cases}
$$
Note that
$$
\lim_{t\to\infty} U_\varepsilon(t)=\frac{a_{\sup}+M(\bar u+\varepsilon)}{b_{\inf}+\chi_2\mu_2-\chi_1\mu_1}.
$$
It then follows that
$$
\bar u=\limsup_{t\to\infty}\|u(t,\cdot;t_0,u_0)\|_\infty\le \frac{a_{\sup}+M(\bar u+\varepsilon)}{b_{\inf}+\chi_2\mu_2-\chi_1\mu_1}
$$
and then
$$
\bar u\le \frac{a_{\sup}}{b_{\inf}+\chi_2\mu_2-\chi_1\mu_1-M}.
$$
The lemma is thus proved.
\end{proof}

Before we state the next lemma,
  let $a_0=\frac{a_{\inf}}{3}$  and
 $L>0$ be a given constant.
 Consider
\begin{equation}
\label{aux-eq1-1}
\begin{cases}
u_t=u_{xx}+a_0 u,\quad x\in (-L,L)\cr
u(t,-L)=u(t,L)=0,
\end{cases}
\end{equation}
and its associated eigenvalue problem
\begin{equation}
\label{aux-eq1-2}
\begin{cases}
u_{xx}+a_0 u=\sigma u,\quad x\in (-L,L)\cr
u(t,-L)=u(t,L)=0.
\end{cases}
\end{equation}
Let $\sigma_{L}$ be the principal eigenvalue of \eqref{aux-eq1-2} and $\phi_L(x)$ be its principal eigenfunction
with $\phi_L(0)=1$. Note that
$$
\sigma_L=-\frac{\pi^2}{4L^2}+a_0
$$
and
$$\phi_L(x)=\cos(\frac{\pi}{2L}x)
\quad {\rm and}\quad 0<\phi_L(x)\le \phi_L(0),\quad \forall x\in (-L,L).
$$
Note also that $u(t,x)=e^{\sigma_{_{L}} t}\phi_L(x)$ is a solution of \eqref{aux-eq1-1}.
Let $u(t,x;u_0)$ be the solution of \eqref{aux-eq1-1} with $u_0\in C([-L,L])$.
Then
\begin{equation}
\label{aux-eq1-3}
u(t,x;\kappa \phi_L)=\kappa e^{\sigma_{_{L}} t}\phi_L(x)
\end{equation}
for all $\kappa\in\R$. Moreover, we have that $\phi_L(x)$ satisfies
\begin{equation}
\label{aux-eq1-2-1}
\begin{cases}
u_{xx}+a_0 u=\sigma_{L} u,\quad x\in (0,L)\cr
u_x(t,0)=u(t,L)=0,
\end{cases}
\end{equation}
and  \eqref{aux-eq1-3}  also holds when $u(t,x; \kappa \phi_L)$ is the solution of
\begin{equation}
\label{aux-eq1-2-1}
\begin{cases}
u_t=u_{xx}+a_0 u,\quad x\in (0,L)\cr
u_x(t,0)=u(t,L)=0
\end{cases}
\end{equation}
with  $u(0,x; \kappa \phi_L)=\kappa\phi_L(x)$ for $x\in [0,L]$.

In the following, fix $T_0>0$ and let $L_0\gg 0$ be such that
$\sigma_{_{L_0}}>0$. Note that $\sigma_{L}$ is increasing as $L$ increases. Choose $N_0\in\NN$ and $\alpha_1>1$ such that
$$
\min\{e^{\sigma_{L_0}T_0}, e^{ \sigma_{N_0L_0} T_0}\cos\big(\frac {\pi}{2N_0}\big)\}\ge \alpha_1.
$$
Then
\begin{equation}
\label{nnew-aux-eq1}
\begin{cases}
u(T_0,0;\kappa \phi_L)=\kappa e^{\sigma_L T_0}\ge \alpha_1 \kappa\cr
u(T_0,x;\kappa \phi_{N_0L})=\kappa e^{\sigma_{N_0L} T_0}\phi_{N_0L}(x)\ge \alpha_1 \kappa\quad {\rm for}\quad |x|\le L
\end{cases}
\end{equation}
for any $L\ge L_0$.

\medskip

\begin{lem}
\label{half-line-lm4}
Consider \eqref{half-line-eq1} and assume {\bf (H1)}.
There is $0<\delta_0^*<M^+=\frac{a_{\sup}}{b_{\inf}+\chi_2\mu_2-\chi_1\mu_1-M}+1$ such that for any $0< \delta\le \delta_0^*$
 and for any $u_0\in C_{\rm unif}^b(\RR^+)$ with
$\delta\le u_0\le M^+$,
\begin{equation}
\label{aux-eq5-1}
\delta\le u(t_0+T_0,x;t_0,u_0)\le M^+\quad \forall\,\, x\in\RR^+, \ \forall\ t_0\in\RR.
\end{equation}
\end{lem}

\begin{proof}
It can be proved by applying properly modified arguments  in \cite[Lemma 3.5]{Salako2018parabolic1}. But the modification is
not trivial.  For the reader's
convenience, we provide some outline of the proof in the following.

First, choose $\alpha_2$ and $\alpha_3$ such that $\alpha_1>\alpha_2>\alpha_3>1$. Consider
\begin{equation}
\label{aux-eq2-1}
\begin{cases}
u_t=u_{xx}+b_\epsilon(x,t) u_x+a_0  u,\quad x\in (-L,L)\cr
u(t,-L)=u(t,L)=0
\end{cases}
\end{equation}
and
\begin{equation}
\label{aux-eq2-1-1}
\begin{cases}
u_t=u_{xx}+b_\epsilon(t,x) u_x+a_0  u,\quad x\in (0,N_0L)\cr
u_x(t,0)=u(t,N_0L)=0,
\end{cases}
\end{equation}
where $|b_\epsilon(x,t)|<\epsilon$  and
$t_0\le t\le t_0+T_0$.
Let $u_{b_\epsilon,L}(x,t
;t_0,u_0)$ be the solution of \eqref{aux-eq2-1}  (resp.  \eqref{aux-eq2-1-1}) with $u_{b_\epsilon,L}(x,t_0;t_0,u_0)=u_0(x)$.
By the similar arguments as those in Step 1 of \cite[Lemma 3.5]{Salako2018parabolic1},  it can be proved that
there is $\epsilon_0>0$ such that for any $L\ge L_0$, $\kappa>0$, and $0\le \epsilon\le \epsilon_0$,
\begin{equation}
\label{aux-eq2-2}
u_{b_\epsilon,L}(t_0+T_0,0;t_0,\kappa \phi_L)\geq \alpha_2\kappa\,\,  ({\rm resp.}\,\,u_{b_\epsilon,L}(t_0+T_0,x;t_0,\kappa \phi_{N_0L})\geq \alpha_2\kappa ,\, {\rm for}\,\,  0\le x\le L)
\end{equation}
provided that $|b_\epsilon(t,x)|<\epsilon$ for $x\in [-L,L]$ (resp. $x\in [0,N_0L])$; and for any
$L\ge L_0$ and $0\le \epsilon\le \epsilon_0$,
\begin{equation}
\label{aux-eq2-2-0}
\begin{cases}
0\le u_{b_\epsilon,L}(t+t_0,x;t_0,\kappa \phi_L)\le e^{a_0 t}\kappa \quad \forall \,\, 0\le t\le T_0,\,\, x\in [-L,L] \cr
 ({\rm resp.} \,  0\le u(t+t_0,x;t_0,\kappa \phi_{N_0L} \le e^{a_0 t} \quad \forall\,\, 0\le t\le T_0, \,\, x\in [0,N_0L]).
 \end{cases}
\end{equation}

Second, consider
\begin{equation}
\label{aux-eq3-1}
\begin{cases}
u_t=u_{xx}+b_\epsilon(t,x) u_x+ u(2 a_0-c(t,x) u),\quad x\in (-L,L)\cr
u(t,-L)=u(t,L)=0,
\end{cases}
\end{equation}
and
\begin{equation}
\label{aux-eq3-1-1}
\begin{cases}
u_t=u_{xx}+b_\epsilon(t,x) u_x+ u(2 a_0-c(t,x) u),\quad x\in (0,N_0L)\cr
u_x(t,0)=u(t,N_0L)=0,
\end{cases}
\end{equation}
where $0\le c(t,x)\le b_{\sup}+\chi_2\mu_2$. Let $u_\epsilon(t,x;t_0,u_0)$ be the solution of \eqref{aux-eq3-1}
(resp. \eqref{aux-eq3-1-1})  with
$u_{\varepsilon}(t_0,x;t_0,u_0)=u_0(x)$.
Assume $L\ge L_0$ and $0\le\epsilon\le\epsilon_0$.
By the similar arguments as those in Step 2 of \cite[Lemma 3.5]{Salako2018parabolic1}, it can be proved that there is $\kappa_0>0$ such that
\begin{equation}
\label{aux-eq3-2}
u_{\varepsilon}(t_0+T_0,0;t_0,\kappa\phi_L)\ge \alpha_3\kappa \,\,\,  {\rm (resp.}\,\,
 u_{\varepsilon}(t_0+T_0,0;t_0,\kappa\phi_{N_0}L)\ge \alpha_3\kappa\,\, {\rm for}\,\, 0\le x\le L)
\end{equation}
for all $0<\kappa \le \kappa_0$.

\smallskip

Third, assume that  $(u(t,x;t_0,u_0), v_1(t,x;t_0,u_0),v_2(t,x;t_0,u_0))$ is the solution
of \eqref{half-line-eq1} on $[t_0,t_0+T_0]$ with $u(t_0,\cdot;t_0,u_0)=u_0(\cdot)\in C_{\rm unif}^b(\RR^+)$.
By the similar arguments as those in Step 3 of \cite[Lemma 3.5]{Salako2018parabolic1},
it can be proved that there is $0<\delta_0\le \kappa_0$ such that for any $u_0\in C_{\rm unif}^b(\RR^+)$ and $x_0\in\RR^+$  with
$0\leq u_0\leq M^+$ and $u_0(x)<\delta_0$ for $x\in\RR^+$, $|x-x_0|\le 3N_0L$, there holds
\begin{equation}
\label{aux-eq4-2}
\begin{cases}
0\le \lambda_1  v_1(t,x;t_0,u_0)\le \frac{a_0}{4\chi_1}\cr
0\le \lambda_2  v_2(t,x;t_0,u_0)\le \frac{a_0}{4\chi_2}\cr
 | \nabla v_1(t,x;t_0,u_0)|<\frac{\epsilon_0}{4 \chi_1} \cr
  | \nabla v_2(t,x;t_0,u_0)|<\frac{\epsilon_0}{4 \chi_2}
  \end{cases}
\end{equation}
for $t_0\le t\le t_0+T_0$, $x\in\RR^+$ with $|x-x_0|\le N_0L$
provided that $L\gg 1$.

\smallskip

Fourth, we claim that  there is $0< \delta_0^*< \min\{\delta_0,M^+\}$ such that for any $0<\delta\le  \delta_0^*$
 and for any $u_0\in C_{\rm unif}^b(\RR^+)$ with
$\delta\le u_0\le M^+$,
\begin{equation}
\label{aux-eq5-1}
\delta\le u(t_0+T_0,x;t_0,u_0)\le M^+\quad \forall\,\, x\in\RR^+.
\end{equation}

Assume that the above claim  does not hold. Then there are $\delta_n\to 0$, $t_{0n}\in\R$, $u_{0n}\in C_{\rm unif}^b(\RR^+)$ with
$\delta_n\le u_{0n}\le M^+$, and $x_n\in\RR^+$ such that
\begin{equation}
\label{aux-eq5-2}
u(t_{0n}+T_0,x_n;t_{0n},u_{0n})<\delta_n.
\end{equation}
For fixed $L\gg 1$, let
$$
D_{0n}=\{x\in\RR^+\,|\, |x-x_n|< 3N_0L,\,\, u_{0n}(x)> \frac{\delta_0}{2}\}.
$$
Without loss of generality, we may assume that $\lim_{n\to\infty} |D_{0n}|$ exists,
where $|D_{0n}|$ is the Lebesgue measure of $D_{0n}$.

Assume that $\lim_{n\to\infty} |D_{0n}|=0$.
Let $\{\tilde{u}_{0n}\}_{n\geq 1}$ be a sequence of elements of   $ C^{b}_{\rm unif}(\RR^+)$ satisfying
$$
\begin{cases}
\delta_n\leq \tilde{u}_{0n}(x)\leq \frac{\delta_{0}}{2}, \quad x\in\RR^+,   |x-x_n|\le 3N_0L\ \  \text{and}\cr
\|\tilde{u}_{0n}(\cdot)-u_{0n}(\cdot)\|_{L^{p}(\RR^+)}\to 0, \quad \forall p>1.
\end{cases}
$$
Let $w_{n}(x,t):=u(t+t_{0n},x;t_{0n},u_{0n}(\cdot))-u(t+t_{0n},x;t_{0n},\tilde{u}_{0n})$ and $v_{i, n}(x,t):=v_i(t+t_{0n},x;t_{0n},u_{0n}(\cdot))-v_i(t+t_{0n},x;t_{0n},\tilde{u}_{0n})$, $i=1,2$.
By the similar arguments as those in Step 4 of \cite[Lemma 3.5]{Salako2018parabolic1},
it can be proved that
\begin{equation}\label{aux-eq5-5}
\lim_{n\to\infty}\sup_{t_{0n}\leq t\leq t_{0n}+T_0}\|w_{n}(t,\cdot)\|_{L^{p}(\RR^+)}=0
\end{equation}
and
\begin{equation}\label{aux-eq5-6}
\lim_{n\to\infty}\sup_{t_{0n}\leq t\leq t_{0n}+ T_0}\|v_{i,n}(t,\cdot)\|_{C^{1,b}_{\rm unif}(\RR^+)}=0,\quad i=1,2.
\end{equation}
By \eqref{aux-eq4-2},  for every $n\geq 1,$
$$
\begin{cases}
0\leq\lambda_1 v_1(t+t_{0n},x;t_{0n},\tilde u_{0n})\leq \frac{a_0}{4\chi_1}\cr
0\leq\lambda_2 v_2(t+t_{0n},x;t_{0n},\tilde u_{0n})\leq \frac{a_0}{4\chi_2}\cr
 |\chi_1 v_{1,x}(t+t_{0n},x;t_{0n},\tilde u_{0n})|\leq \frac{\varepsilon_0}{4}\cr
|\chi_2 v_{2,x}(t+t_{0n},x;t_{0n},\tilde u_{0n})|\leq \frac{\varepsilon_0}{4}\\
\end{cases}
$$
for all $0\leq t\leq T_0$ and $x\in\RR^+$ with $|x-x_n|\le N_0L$.
This together with  \eqref{aux-eq5-6} implies that, for $n\gg 1$, there holds
$$
\begin{cases}
0\leq \chi_1\lambda_1 v_1(t+t_{0n},x;t_{0n}, u_{0n})\leq \frac{a_0}{2}\cr
0\leq \chi_2\lambda_2 v_2(t+t_{0n},x;t_{0n}, u_{0n})\leq \frac{a_0}{2}\cr
 |\chi_1 v_{1,x} (t+t_{0n},x;t_{0n}, u_{0n}(\cdot))|\leq \frac{\varepsilon_0}{2}\cr
  |\chi_2 v_{2,x} (t+t_{0n},x;t_{0n}, u_{0n}(\cdot))|\leq \frac{\varepsilon_0}{2}
\end{cases}
$$
for all $0\leq t\leq  T_0$ and $x\in\RR^+$ with $|x-x_{n}|\le N_0L$.
Hence
$$
\begin{cases}
|\chi_1 v_{1,x}(t+t_{0n},x;t_{0n},u_{0n})-\chi_2 v_{2,x}(t+t_{0n},x;t_{0n},u_{0n})|\le \varepsilon_0\cr
|\chi_1\lambda_1 v_1(t+t_{0n},x;t_{0n},u_{0n})-\chi_2\lambda_2 v_2(t+t_{0n},x;t_{n},u_{0n})\le a_0
\end{cases}
$$
for all $0\le t\le T_0$ and $x\in\RR^+$ with $|x-x_{n}|\le N_0L$.

Let
$$b_n(t,x)=-\chi_1 v_{1,x}(t+t_{0n},x;t_{0n},u_{0n})+\chi_2 v_{2,x}(t+t_{0n},x;t_{0n},u_{0n})
$$
and
$$
u_n(t,x)=u(t+t_{0n},x+x_n;t_{0n},u_{0n}).
$$
In the case $x_{0n}>L$, $u_n(t,x)$ satisfies
$$
\begin{cases}
u_t\ge u_{xx}+b_n(t,x) u_x+u(2a_0-(b_{\sup}+\chi_2\mu_2)u),\quad -L<x<L\cr
u(t,-L)>0, \,\, u(t,L)>0\cr
u(t_{0n},x)\ge \delta_n,\quad x\in [-L,L].
\end{cases}
$$
In the case $x_n\le L$, $u_n(t,x)$ satisfies
$$
\begin{cases}
u_t\ge u_{xx}+b_n(t,x) u_x+u(2a_0-(b_{\sup}+\chi_2\mu_2)u),\quad -x_n<x<N_0L\cr
u_x(t,-x_n)=0, \,\, u(t,L)>0\cr
u(t_{0n},x)\ge \delta_n,\quad x\in [-x_n,N_0L].
\end{cases}
$$
In either case,  it follows from the arguments of \eqref{aux-eq3-2} that
$$u(T+t_{0n},x_n;t_{0n}, u_{0n})=u_n(T_0,0) >\delta_{n},
$$
which is a contradictions. Hence $\lim_{n\to\infty} |D_{0n}|\not =0$.

Without loss of generality, we may then assume that $\inf_{n\ge 1} |D_{0n}|>0$ and there is $L>0$ such that
$$
\inf_{n\ge 1}|\{x\in\RR^+\,|\, x\in D_{0n}\cap [x_n-3N_0L,x_n+3N_0L]\}|>0.
$$
By the similar arguments as those in Step 4 of \cite[Lemma 3.5]{Salako2018parabolic1}, it can be proved that there is $0<\tilde T_0<T_0$ such that
$$
\inf_{n\ge 1}\|u(t_{0n}+\tilde T_0,\cdot;t_{0n},u_{0n})\|_{C([x_n-3N_0L,x_n+3N_0L]\cap [0,x_n+3N_0L])}>0.
$$
Moreover, we  might suppose that $x_n\to x^*\in [0,\infty]$ and
$u( \tilde T_0+t_{0n}, x_n+\cdot ;t_{0n},u_{0n}(\cdot+x_n))\to u^{*}_{0}(\cdot)$ locally uniformly on $(-x^*,\infty)$ and $\|u^{*}_{0}\|_{C^0([-3N_0L,3N_0L]\cap [-x^*,3N_0L])}>0$. Also, we might assume that
$(u(t+t_{0n},x_n+\cdot;t_{0n},u_{0n}),v_1( t+t_{0n},x_n+\cdot;t_{0n},u_{0n}),v_2(t+t_{0n},\cdot+x_n;t_{0n},u_{0n}))\to (u^{*}(t,x),v_1^{*}(t,x),v_2^*(t,x))$ locally uniformly on $[\tilde T_0,\infty)\times (-x^*,\infty)$, $a(t,x+x_{n})\to a^{*}(t,x)$, and  $b(t,x+x_{n})\to b^{*}(t,x)$, where $(u^{*},v_1^{*},v_2^{*})$ satisfies
\begin{equation*}
\begin{cases}
u^{*}_{t}=u^{*}_{xx}-\chi_1(u^{*} v^{*}_{1,x})_x+\chi_2 (u^* v^*_{2,x})_x+(a^{*}(t,x)-b^{*}(t,x)u^{*})u^{*},\quad -\infty<x<\infty\cr
0=v^{*}_{1,xx}-\lambda_1 v_1^*+\mu_1 u^{*},\quad -\infty<x<\infty \cr
0=v^{*}_{2,xx}-\lambda_2 v_2^*+\mu_2 u^{*},\quad -\infty<x<\infty \cr
u^*(\tilde T_0,\cdot)=u^{*}_{0}
\end{cases}
\end{equation*}
in the case $x^*=\infty$,  and satisfies
\begin{equation*}
\begin{cases}
u^{*}_{t}=u^{*}_{xx}-\chi_1(u^{*} v^{*}_{1,x})_x+\chi_2 (u^* v^*_{2,x})_x+(a^{*}(t,x)-b^{*}(t,x)u^{*})u^{*},\quad - x^*<x<\infty\cr
0=v^{*}_{1,xx}-\lambda_1 v_1^*+\mu_1 u^{*},\quad -x^*<x<\infty \cr
0=v^{*}_{2,xx}-\lambda_2 v_2^*+\mu_2 u^{*},\quad - x^*<x<\infty \cr
u^*_x(t,-x^*)=v_{1,x}^*(t,-x^*)=v_{2,x}^*(t,-x^*)=0\cr
u^*(\tilde T_0,\cdot)=u^{*}_{0}
\end{cases}
\end{equation*}
in the case $x^*<\infty$.
Since $\|u^*_0\|_{\infty}>0$ and $u^*(t,x)\geq 0$, it follows from comparison principle for parabolic equations that $u^{*}(t,x)>0$ for every $x\in (-x^*,\infty)$ and $t\in (\tilde T_0, \infty)$. In particular $u^{*}(T_0,0)>0$. Note by \eqref{aux-eq5-2} that we must have $u^{*}(T_0,0)=0$, which is a contradiction. Hence the claim \eqref{aux-eq5-1} holds. The lemma is thus proved.
\end{proof}

\section{Proofs of the main results}

In this section, we prove Theorems \ref{half-line-thm1}-\ref{half-line-thm3} and Theorem \ref{whole-line-thm}. We mainly provide the proof for Theorems \ref{half-line-thm1}-\ref{half-line-thm3}. Theorem \ref{whole-line-thm} can be proved by the similar arguments of Theorems \ref{half-line-thm1}-\ref{half-line-thm3}.

\subsection{Global existence}

In this subsection, we prove Theorem \ref{half-line-thm1} for the global existence of solutions of \eqref{half-line-eq1} with nonnegative initial functions.

\begin{proof}[Proof of Theorem \ref{half-line-thm1}]
By Lemma \ref{half-line-lm1},
for any $t_0\in\RR$ and any nonnegative function $u_0\in C^{b}_{\rm unif}(\RR^+)$, \eqref{half-line-eq1} has a unique solution
    $(u(t,x;t_0,u_0),v_1(t,x;t_0,u_0)$, $v_2(t,x;t_0,u_0))$ with $u(t_0,x$; $t_0,u_0)=u_0(x)$ defined on $[t_0, t_0+T_{\max})$.
Moreover, if $T_{\max}<\infty$, then
$$
\limsup_{t\to T_{\max}} \|u(t_0+t,\cdot;t_0,u_0)\|_\infty=\infty.
$$

Let $C_0=C_0(u_0)$ be as in Lemma \ref{half-line-lm2}.  For any give $0<T<T_{\max}$, let
 $$\mathcal{E}^{T}=C^{b}_{\rm unif}([0,T]\times \RR^+)$$
  endowed with the norm
\begin{equation}\label{global-exist-eq-r001}
\|u\|_{\mathcal{E}^T}:=\sum_{k=1}^{\infty}\frac{1}{2^k}\|u\|_{L^{\infty}([0,T]\times [0,k])}.
\end{equation}
 Consider the subset $\mathcal{E}$ of $\mathcal{E}^T$ defined by
 $$\mathcal{E}:=\{u\in  C_{\rm unif}^b([0,T]\times \RR^+)\,|\, u(0,\cdot)=u_0, 0\leq u(x,t)\leq C_0, x\in\RR^+, 0\leq t\leq T\}.
  $$
It is clear that
\begin{equation}\label{global-exist-eq-r002}
\|u\|_{\mathcal{E}^T}\leq C_{0},  \quad \forall\ u\in\mathcal{E}.
\end{equation}
Moreover,  $\mathcal{E}$ is a closed bounded and convex subset of $\mathcal{E}^T$.  We shall show that $u(t_0+\cdot,\cdot;t_0,u_0)\in \mathcal{E}$.

To this end, for any given $u\in \mathcal{E}$, let $v_i(t,x;u)$ be the solution of
$$
\begin{cases}
0=v_{i,xx}-\lambda_i v+\mu_i u(t,x),\quad x\in\RR^+\cr
\frac{\p v_i}{\p n}(t,0)=0.
\end{cases}
$$
Let  $U(x,t;u)$ be the solution of the initial value problem
\begin{equation}\label{global-exist-eq004}
\begin{cases}
U_{t}=\Delta U+\nabla\big(\chi_2 v_2(t,x;u)-\chi_1v_1(t,x;u)\big)\nabla U\\
\qquad\,\,\,  +U\Big(a(t,x)+(\chi_2\lambda_2v_2(t,x;u) -\chi_1\lambda_1v_1(t,x;u)) -(b(t,x)+\chi_2\mu_2-\chi_1\mu_1)U\Big ),\quad x\in\RR^+\\
\frac{\p U}{\p n}(t,0)=0\\
U(t_0,\cdot;u)=u_0(\cdot).
\end{cases}
\end{equation}
By Lemma \ref{half-line-lm2}, we have
$$
(\chi_2 \lambda_2 v_2-\chi_1 \lambda_1 v_1)(x,t)\leq M C_{0}\quad \forall\, t\in [t_0,t_0+T].
$$
Hence for $t\in (t_0,t_0+T]$,
\begin{equation}\label{global-exist-eq004-1}
\begin{cases}
U_{t}\le \Delta U+\nabla\big(\chi_2 v_2(t,x;u)-\chi_1v_1(t,x;u)\big)\nabla U\\
\qquad\,\,\,  +U\Big(a_{\sup}+C_0M -(b_{\inf}+\chi_2\mu_2-\chi_1\mu_1)U\Big ),\quad x\in\RR^+\\
\frac{\p U}{\p n}(t,0)=0\\
U(t_0,\cdot;u)=u_0(\cdot).
\end{cases}
\end{equation}
Observe that $U\equiv C_0$ is a super-solution of \eqref{global-exist-eq004-1}. Hence
by comparison principle for parabolic equations,
we have
$$
U(t,x;u)\le C_0 \quad t\in [t_0,t_0+T], \,\,\, x\in\RR^+.
$$
Therefore, $U(\cdot,\cdot;u)\in \mathcal{E}$.

 By the similar  arguments  as those in \cite[Lemma 4.3]{Salako2016spreading}, the mapping $\mathcal{E}\ni u\mapsto U(\cdot,\cdot;u)\in\mathcal{E}$ is continuous and compact, and then by Schauder's fixed theorem, it has a fixed point $u^*$. Clearly $(u^*(\cdot,\cdot),v_1(\cdot,\cdot;u^*),v_2(\cdot,\cdot;u^*))$ is a classical solution of \eqref{half-line-eq1}. Thus, by Lemma \ref{half-line-lm1}, we have $$
 u(t,x;t_0,u_0)=u^*(t,x)\le C_0\quad \forall\, t\in [t_0,t_0+T], \,\, x\in\RR^+.
 $$
 Since $0<T<T_{\max}$ is arbitrary, by Lemma \ref{half-line-lm1} again, we have $T_{\max}=\infty$ and
$$
0\le u(t,x;t_0,u_0)\le  C_0\quad \forall \,\, t\in [t_0,\infty), \,\, x\in [0,\infty).
$$
{ Moreover, by Lemma \ref{half-line-lm3},
$$
\limsup_{t\to\infty} \|u(t,\cdot;t_0,u_0)\|_\infty\le \frac{a_{\sup}}{b_{\inf}+\chi_2\mu_2-\chi_1\mu_1-M}.
$$}
 Theorem \ref{half-line-thm1} then follows.
\end{proof}

\subsection{Persistence}

In this subsection, we prove Theorem \ref{half-line-thm2} on the persistence of solutions of \eqref{half-line-eq1} with strictly positive initial functions.

\begin{proof}[Proof of Theorem \ref{half-line-thm2}]
(1) Assume that {\bf (H1)} holds.
 Fix $T_0>0$. Let $\delta_0^*$ and $M^+$  be as in Lemma \ref{half-line-lm4}.
   For any $u_0\in C_{\rm unif}^b(\RR^+)$ with $\inf_{x\in\RR^+}u_0(x)>0$, by Theorem \ref{half-line-thm1},
   $$
   u(t,x;t_0,u_0)\le C(u_0)\quad \forall\,\, t\ge t_0,\,\, x\in\RR^+.
   $$
  By Lemma \ref{half-line-lm3}, there is $T_1>0$ such that
  $$
  u(t,x;t_0,u_0)\le M^+\quad \forall\, t\ge t_0+T_1,\,\, x\in\RR^+.
  $$
  Observe that
  $$
  \inf_{x\in\RR^+} u(t_0+T_1,x;t_0,u_0)>0.
  $$
  Then there is $0<\delta\le \delta_0^*$ such that
  $$
  \delta\le u(t_0+T_1,x;t_0,u_0)\le M^+\quad \forall\,\, x\in\RR^+.
  $$
  By Lemma \ref{half-line-lm4},
  $$
  \delta\le u(t_0+T_1+nT_0,x;t_0,u_0)\le M^+\quad \forall\,\, n\in\NN,\,\, x\in\RR^+.
  $$
  This implies that there is $m(u_0)>0$ such that
  $$
  m(u_0)\le u(t,x;t_0,u_0)\le M^+\quad \forall\,\, t\ge t_0,\, \, x\in\RR^+.
  $$

  (2) Assume that {\bf (H2)} holds.  Let
  $$
M_0= \frac{a_{\sup}}{b_{\inf}-\chi_1\mu_1+\chi_2\mu_2-M}
$$
and
$$
m_0= \frac{a_{\inf}\big(b_{\inf}-(1+\frac{a_{\sup}}{a_{\inf}})\chi_1\mu_1+\chi_2\mu_2-M\big)}{(b_{\inf}-\chi_1\mu_1+\chi_2\mu_2-M)
(b_{\sup}-\chi_1\mu_1 + \chi_2\mu_2)}.
$$
By {\bf (H2)}, $m_0>0$.
  For given $u_0\in C^{b}_{\rm unif}(\RR^+)$ with $\inf_{x\in\RR^+}u_0(x)> 0 $,  define
$$
\underline{u}:=\liminf_{t\to\infty}\inf_{x\in\RR^+}u(x,t+t_0;t_0,u_0)\quad \text{and}\quad \overline{u}:=\limsup_{t\to\infty}\sup_{x\in\RR^+}u(x,t+t_0;t_0,u_0).
$$
If suffices to prove that
$$
m_0\le \underbar{u}\le \bar u\le M_0.
$$

By (1), $\underbar{u}>0$.
Using the definition of limsup and liminf, we have that for every $0<\varepsilon<\underbar{u}$, there is $T_{\varepsilon}>0$ such that
$$
\underline{u}-\varepsilon\leq u(x,t;t_0,u_0)\leq \overline{u}+\varepsilon\quad \forall\ x\in\RR^+,\ \forall\ t\geq t_0+T_{\varepsilon}.
$$
Hence, it follows from comparison principle for  elliptic equations, that
\begin{equation}\label{asym-eq03}
\mu_i(\underline{u}-\varepsilon)\leq \lambda_iv_i(x,t;t_0,u_0)\leq \mu_i(\overline{u}+\varepsilon), \forall\ x\in\RR^+,\ \ \forall \ t\ge T_{\varepsilon},\ i=1,2.
\end{equation}
We then have
\begin{align*}
u_t&=u_{xx}+(\chi_2v_2 - \chi_1v_1)_x u_x +u\big(a(t,x)-\chi_1 \lambda_1 v_1+\chi_1\mu_2 v_2-(b(t,x)-\chi_1\mu_1+\chi_2\mu_2)u\big)\\
&\ge u_{xx}+(\chi_2v_2 - \chi_1v_1)_x u_x +u\big(a_{\inf}-\chi_1 \mu_1(\bar u+\varepsilon)+\chi_2\mu_2 (\underbar{u}-\varepsilon)-(b_{\sup}-\chi_1\mu_1+\chi_2\mu_2)u\big)
\end{align*}
for $t\ge t_0+T_\varepsilon$.
This together with comparison principle for parabolic equations implies that
\begin{align*}
\underbar{u}&\ge   \frac{a_{\inf}-\chi_1 \mu_1(\bar u+\varepsilon)+\chi_2\mu_2 (\underbar{u}-\varepsilon)}
{(b_{\sup}-\chi_1\mu_1+\chi_2\mu_2)}\\
&\ge \frac{a_{\inf}-\chi_1 \mu_1(\bar u+\varepsilon)}
{(b_{\sup}-\chi_1\mu_1+\chi_2\mu_2)}.
\end{align*}
Let $\varepsilon\to 0$, we have
$$
\underbar{u} \ge \frac{a_{\inf}-\chi_1 \mu_1\bar u}
{(b_{\sup}-\chi_1\mu_1 + \chi_2\mu_2)}.
$$
By Lemma \ref{half-line-lm3},
$$
\bar u\le M_0=\frac{a_{\sup}}{b_{\inf}-\chi_1\mu_1+\chi_2\mu_2-M}.
$$
It then follows that
$$
\underbar{u}\ge m_0= \frac{a_{\inf}\big(b_{\inf}-(1+\frac{a_{\sup}}{a_{\inf}})\chi_1\mu_1+\chi_2\mu_2-M\big)}{(b_{\inf}-\chi_1\mu_1+\chi_2\mu_2-M)
(b_{\sup}-\chi_1\mu_1 +\chi_2\mu_2)}.
$$
\end{proof}

\subsection{Positive entire solutions}

In this subsection, we prove Theorem \ref{half-line-thm3} on the existence, uniqueness, and stability of strictly positive entire solutions of
\eqref{half-line-eq1}

We first prove Theorem \ref{half-line-thm3}(1).

\begin{proof}[Proof of Theorem \ref{half-line-thm3}(1)]
It can be proved by applying properly modified arguments  in  \cite[Theorem 1.4 (iii)]{Salako2018parabolic2}. For the reader's
convenience, we provide some outline of the proof.

\smallskip

First, Let
$$
M^+= \frac{a_{\sup}}{b_{\inf}+\chi_2\mu_2-\chi_1\mu_1-M}+1.
$$
By Lemma \ref{half-line-lm4},
there is $0<\delta_0^*<M^+$ such that for any $0< \delta\le \delta_0^*$
 and for any $u_0\in C_{\rm unif}^b(\RR^+)$ with
$\delta\le u_0\le M^+$,
\begin{equation}
\label{new-gl-eq1}
\delta\le u(t_0+T_0,x;t_0,u_0)\le M^+\quad \forall\,\,  x\in [0,\infty),\,\, t_0\in\RR.
\end{equation}

\smallskip

Next, let
$$
u_n(t,x)=u(t-nT,x;-nT,\delta_0^*)\quad \forall\,\, t\ge -nT,\,\, x\in\RR^+
$$
Then there is $n_k\to \infty$ and $u^*(t,x)$ such that
$$
\lim_{n\to\infty} u_n(t,x)=u^*(t,x)
$$
locally uniformly on $\RR\times [0,\infty)$. It can then be verified that
$(u,v_1,v_2)=(u^*(t,x),v_1^*(t,x),v_2^*(t,x))$ is a strictly positive entire solution of \eqref{half-line-eq1},
where $v_i^*(t,x)$ satisfies
$$
\begin{cases}
0=v_{i,xx}-\lambda_i v+\mu_i u^*(t,x),\quad 0<x<\infty\cr
v_x(t,0)=0
\end{cases}
$$
for $i=1,2$.

\smallskip

Now, we show that, if $a(t+T,x)\equiv a(t,x)$ and $b(t+T,x)\equiv b(t,x)$, then \eqref{half-line-eq1} has a time $T-$ periodic positive solution.
To this end, choose $T_0=T$ and  let
 $$\mathcal{E}=\{u\in C^{b}_{\rm unif}(\RR^+)\ |\ \delta_{0}^*\leq u_{\inf}\leq u_{\sup}\leq M^+\}$$
  endowed with the open compact topology. For any $u_0\in \mathcal{E}$, define
  $$
  \mathcal{P}u_0=u(T,\cdot;0,u_0).
  $$
  By \eqref{new-gl-eq1}, $\mathcal{P}u_0\in \mathcal{E}$.
  By the similar arguments as those in \cite[Theorem 1.4 (iii)]{Salako2018parabolic2}, it can be proved that
  $\mathcal{P}:\mathcal{E}\to \mathcal{E}$ is a continuous and compact map. Then
   Schauder's fixed theorem implies that there is $u^*\in E$ such that $u(T,\cdot;0,u^*)=u^*$.  Clearly $(u(\cdot,\cdot;0,u^*),v_1(\cdot,\cdot;0,u^*), v_2(\cdot,\cdot;0,u^*))$ is a $T-$periodic solution of \eqref{one-free-boundary-eq}
    and hence is a  positive entire solution.  Theorem \ref{half-line-thm3}(1) is thus proved.
\end{proof}

Next, we  prove Theorem \ref{half-line-thm3} (2).

 \begin{proof}[Proof of Theorem \ref{half-line-thm3} (2)]
 (i)
 First,  note that, by Lemmas \ref{half-line-lm2} and \ref{half-line-lm3}, we only need to prove the statement for
 $u_0\in C_{\rm unif}^b(\RR^+)$ satisfying
 \begin{equation}
 \label{asym-eq0}
 0<\inf_{x\in\RR^+} u_0(x)\le \sup_{x\in\RR^+}u_0(x)<\frac{a_{\sup}+1}{b_{\inf}+\chi_2\mu_2-\chi_1\mu_1-M}.
 \end{equation}
 Note also that
$(u(t,x),v_1(t,x),v_2(t,x))=(u^*(t), v_1^*(t),v_2^*(t))$
 is a strictly positive periodic solution of \eqref{half-line-eq1}, where $u=u^*(t)$ is the unique positive $T$-periodic solution of the ODE
\begin{equation}
\label{ode-eq1}
u^{'}=(a(t)-b(t)u)u,
\end{equation}
and
$v_1^*(t)=\frac{\mu_1}{\lambda_1} u^*(t)$, $v_2^*(t)=\frac{\mu_2}{\lambda_2} u^*(t)$.
It then suffices to prove that, for any given $t_0\in\RR$ and $u_0\in C_{\rm unif}^b(\RR^+)$ satisfying \eqref{asym-eq0},
\begin{equation}
\label{asym-eq1}
\lim_{t\to\infty}\|u(t+t_0,\cdot;t_0,u_0)-u^*(t+t_0)\|_\infty=0.
\end{equation}

To prove \eqref{asym-eq1}, for given $t_0\in\RR$ and $u_0\in C_{\rm unif}^b(\RR^+)$ with $\inf_{x\in\RR^+}u_0(x)>0$,
define
\begin{equation*}
 U(t,x)=\frac{u(t,x;t_0,u_0)}{u^*(t)},\,\,\, V_1(t,x)=\frac{v_1(t,x;t_0,u_0)}{v_1^*(t)},\,\,\, V_2(t,x)=\frac{v_2(t,x;t_0,u_0)}{v_1^*(t)}.
\end{equation*}
It then suffice to prove that
\begin{equation}
\label{asym-eq2}
\lim_{t\to\infty} \|U(t,\cdot)-1\|_\infty=0.
\end{equation}
We claim that  for any $\varepsilon>0$, there are
$T_{\varepsilon,n}$ $(n=1,2,\cdots)$ such that  for any $t\geq T_{\varepsilon,n}$,
\begin{equation}\label{asym-eq2-1}
\|U(t+t_0,\cdot)-1\|_{\infty}\leq \Big(\frac{K}{b_{\inf}+\chi_2\mu_2-\chi_1\mu_1}\Big)^{n}\frac{a_{\sup}+1}{(b_{\inf}+\chi_2\mu_2-\chi_1\mu_1-M)u^*_{ \inf}} +\varepsilon ,
\end{equation}
where $K$ is defined in \eqref{k-eq}.
Note that {\bf (H3)} implies that  $\frac{K}{b_{\inf}+\chi_2\mu_2-\chi_1\mu_1}<1$.
 Hence
 \eqref{asym-eq2} follows from \eqref{asym-eq2-1}. In the following, we prove \eqref{asym-eq2-1} by induction.

First, by direct calculation, we have
\begin{align}\label{asym-eq3}
U_t=&\Delta U -\chi_1 U_x v_{1,x}+\chi_2 U_x v_{2,x}+ \chi_1\mu_1 U(V_1-1)u^*(t)-\chi_2\mu_2 U(V_2-1) u^*(t)\nonumber \\
&+ (b(t)-\chi_1\mu_1+\chi_2\mu_2)U(1-U)u^*(t).
\end{align}
Obverse that $V_1(t,x)$  satisfies
$$
\begin{cases}
0=(V_1-1)_{xx}-\lambda_1 (V_1-1)+\lambda_1 (U(t,x)-1),\quad x\in (0,\infty)\cr
(V_{1}-1)_{x}(t,0)=0
\end{cases}
$$
and $V_2(t,x)$ satisfies
$$
\begin{cases}
0=(V_{2}-1)_{xx}-\lambda_2 (V_2-1)+\lambda_2 (U(t,x)-1),\quad x\in (0,\infty)\cr
(V_{2}-1)_{x}(t,0)=0
\end{cases}
$$
for $t\ge t_0$.  Then by Lemma \ref{half-line-lm2}(2),
$$
\|\chi_2 \mu_2 (V_2(t,\cdot)-1)-\chi_1 \mu_1( V_1(t,\cdot)-1)\|_\infty \leq K \|U(t,\cdot)-1\|_\infty \quad \forall\, t\in [t_0,\infty),
$$
where $K$ is as in \eqref{k-eq}.  Observe also that
$$
\limsup_{t\to\infty} \|U(t,\cdot)-1\|_\infty\le \frac{a_{\sup}}{\big(b_{\inf}+\chi_2\mu_2-\chi_1\mu_1-M\big)u^*_{\inf} }.
$$
Let
$$
\tilde M_1= \frac{a_{\sup}+1}{\big(b_{\inf}+\chi_2\mu_2-\chi_1\mu_1-M\big)u^*_{\inf}} .
$$
Then there is $\tilde T_{1,\varepsilon}>0$ such that
\begin{align*}
U_t\le \Delta U -\chi_1 U_x v_{1,x}+\chi_2 U_x v_{2,x}+\Big[K \tilde M_1  U + (b(t)-\chi_1\mu_1+\chi_2\mu_2)U(1-U)\Big]u^*(t)
\end{align*}
and
\begin{align*}
U_t\ge \Delta U -\chi_1 U_x v_{1,x}+\chi_2 U_x v_{2,x}-\Big[K \tilde M_1  U + (b(t)-\chi_1\mu_1+\chi_2\mu_2)U(1-U)\Big]u^*(t)
\end{align*}
for $t\ge \tilde T_{1,\varepsilon}$.

Next, it is not difficult to see that
$$
\inf_{x\in\RR^+} U(t,x)>0\quad \forall \,\, t\ge t_0.
$$
Let $\bar U(t)$  be the solution of
$$
\begin{cases}
U_t=\Big[K  \tilde M_1 U + ({b_{\inf}}-\chi_1\mu_1+\chi_2\mu_2)U(1-U)\Big]u^*(t)\cr
{ U(\tilde T_{1,\epsilon})}= \max\{ \|U(\tilde T_{1,\epsilon},\cdot)\|_\infty, 1+\frac{K}{b_{\inf}+\chi_2\mu_2-\chi_1\mu_1}\tilde M_1\}
\end{cases}
$$
and $\underbar{U}(t)$ be the solution of
$$
\begin{cases}
U_t=\Big[- K \tilde M_1  U + (b_{\inf}-\chi_1\mu_1+\chi_2\mu_2)U(1-U)\Big]u^*(t)\cr
{ U(\tilde T_{1,\epsilon})}=\inf\{\inf_{x\in\RR^+}U(\tilde T_{1,\epsilon},x),1\}.
\end{cases}
$$
Then
$$
\lim_{t\to\infty} \bar U(t)= 1+\frac{K}{{b_{\inf}}+\chi_2\mu_2-\chi_1\mu_1}\tilde M_1
$$
and
$$
\lim_{t\to\infty} \underbar{U}(t)=1-\frac{K}{b_{\inf}+\chi_2\mu_2-\chi_1\mu_1} \tilde M_1 .
$$
Moreover, it is not difficult to see that
$$
\underbar{U}(t)\le U(t,x)\le \bar U(t)\quad \forall \, t\ge \tilde T_{1,\varepsilon}.
$$
It then follows that for any given $\varepsilon>0$,  there is $T_{1,\varepsilon}\ge \tilde T_{1,\varepsilon}$ such that \eqref{asym-eq2-1} holds with $n=1$.

Next, assume that \eqref{asym-eq2-1} holds for $n=k$. For fixed $\varepsilon>0$, let $\tilde \varepsilon= \frac{\varepsilon}{2} \frac{b_{\inf}+\chi_2\mu_2-\chi_1\mu_1-M}{K}$. Then, by the similar arguments as in the above, we have
\begin{align*}
&\|\chi_2\mu_2 (V_2(t,\cdot)-1)-\chi_1\mu_1 (V_1(t,\cdot)-1)\|_\infty\\
&
\le  K \Big(\frac{K}{b_{\inf}+\chi_2\mu_2-\chi_1\mu_1}\Big)^{k}\frac{(a_{\sup}+1)}{(b_{\inf}+\chi_2\mu_2-\chi_1\mu_1-M)u^*_{ \inf}} + K\tilde \varepsilon
\end{align*}
for $t\ge T_{\tilde \varepsilon,k}$. Let
$$
\tilde M_k=\Big(\frac{K}{b_{\inf}+\chi_2\mu_2-\chi_1\mu_1}\Big)^{k}\frac{(a_{\sup}+1)}{(b_{\inf}+\chi_2\mu_2-\chi_1\mu_1-M)u^*_{ \inf}} +\tilde \varepsilon.
$$
Then for $t\ge T_{\tilde\varepsilon,k}$,
\begin{align*}
U_t\le \Delta U -\chi_1 U_x v_{1,x}+\chi_2 U_x v_{2,x}+\Big[K \tilde M_k  U + (b(t)-\chi_1\mu_1+\chi_2\mu_2)U(1-U)\Big]u^*(t)
\end{align*}
and
\begin{align*}
U_t\ge \Delta U -\chi_1 U_x v_{1,x}+\chi_2 U_x v_{2,x}-\Big[K \tilde M_k  U + (b(t)-\chi_1\mu_1+\chi_2\mu_2)U(1-U)\Big]u^*(t).
\end{align*}
Moreover, we have
$$
\underbar{ U}_{k}(t)\le U(T,x)\le \bar U_k(t)\quad \forall\, t\ge T_{\tilde \varepsilon,k},\,\, x\in\RR^+,
$$
where
$\bar U_k(t)$ is the solution of
$$
\begin{cases}
U_t=\Big[K \tilde M_k U + ({b_{\inf}}-\chi_1\mu_1+\chi_2\mu_2)U(1-U)\Big]u^*(t)\cr
 U(T_{\tilde\varepsilon, k})= \max\{ \sup_{x\in\RR^+}U(T_{\tilde \varepsilon,k}, x), 1+\frac{K}{b_{\inf}+\chi_2\mu_2-\chi_1\mu_1}\tilde M_k\}
\end{cases}
$$
and $\underbar{U}_k(t)$ be the solution of
$$
\begin{cases}
U_t=\Big[- K \tilde M_k  U + (b_{\inf}-\chi_1\mu_1+\chi_2\mu_2)U(1-U)\Big]u^*(t)\cr
 U(T_{\tilde\varepsilon,k})=\inf\{\inf_{x\in\RR^+}U(T_{\tilde\varepsilon,k},x),1\}.
\end{cases}
$$
Note that
$$
\lim_{t\to\infty}\bar U_k(t)=1+\frac{K}{{b_{\inf}}+\chi_2\mu_2-\chi_1\mu_1-M}\tilde M_k
$$
and
$$
\lim_{t\to\infty}\underbar{U}_k(t)=1-\frac{K}{b_{\inf}+\chi_2\mu_2-\chi_1\mu_1-M}\tilde M_k
$$
It then follows that there is $T_{\varepsilon,k+1}$ such that
$$
\|U(t,\cdot)-1\|_\infty\le \Big(\frac{K}{b_{\inf}+\chi_2\mu_2-\chi_1\mu_1}\Big)^{k+1}\frac{(a_{\sup}+1)}{(b_{\inf}+\chi_2\mu_2-\chi_1\mu_1-M)u^*_{ \inf}} + \varepsilon.
$$

By induction, \eqref{asym-eq2-1} holds for any $n\ge 1$.  Theorem \ref{half-line-thm3}(2)(i) is thus proved.

\medskip

(ii)  It can be proved by combing the arguments in (i) and properly modified arguments in \cite[Theorem 1.5]{Salako2018parabolic2}.
We do not provide the proof in this paper.
\end{proof}

\begin{proof}[Proof of Theorem \ref{whole-line-thm}]
(1) It follows from the similar arguments as those in Theorem \ref{half-line-thm1}.

(2) It follows  from the similar arguments as those in Theorem \ref{half-line-thm2}.

(3) It follows  from the similar arguments as those in Theorem \ref{half-line-thm3}(1).

(4) It follows  from the similar arguments as those in Theorem \ref{half-line-thm3}(2).
\end{proof}

\section*{Acknowledgments}
 The authors would like to thank Prof. Yihong Du for helpful comments and suggestions on the derivation of the free boundary condition
  in \eqref{one-free-boundary-eq}.

\end{document}